\documentclass[final,3p]{elsarticle}
\usepackage{lineno,hyperref}
\modulolinenumbers[5]
\usepackage[all,cmtip]{xy}
\usepackage{array}
\usepackage{amsmath, amssymb, amsthm}
\usepackage{graphics}
\usepackage{algorithm}
\usepackage{algorithmicx}
\usepackage{booktabs}
\usepackage[dvipsnames]{xcolor}
\usepackage{cancel}
\usepackage{caption}
\usepackage{subcaption}
\usepackage[noend]{algpseudocode}
\usepackage{ifthen}
\usepackage{bm}
\usepackage{pdflscape}
\usepackage{pifont}
\usepackage[normalem]{ulem}
\usepackage{multirow}
\def\bfb{{\boldsymbol{b}}}
\def\bfc{{\mathbf{c}}}

\def\bfp{{\boldsymbol{p}}}
\def\bfq{{\boldsymbol{q}}}

\def\bfx{{\boldsymbol{x}}}

\def\bfu{{\boldsymbol{u}}}
\def\bfv{{\boldsymbol{v}}}
\def\bfw{{\boldsymbol{w}}}

\def\bfC{{\bm{C}}}

\newtheorem{theorem}{Theorem}
\newtheorem{corollary}[theorem]{Corollary}
\newtheorem{proposition}[theorem]{Proposition}
\newtheorem{lemma}[theorem]{Lemma}
\theoremstyle{definition}

\newtheorem{example}{Example}
\newtheorem{remark}{Remark}

\graphicspath{{./figs/}}

\begin{document}

\begin{frontmatter}
\title{Detecting affine equivalences between certain types of parametric curves, in any dimension}

\author[a]{Juan Gerardo Alc\'azar}
\ead{juange.alcazar@uah.es}
\author[b]{H\"usn\"u An{\i}l \c{C}oban}
\ead{hacoban@ktu.edu.tr}
\author[b]{U\u{g}ur G\"oz\"utok}
\ead{ugurgozutok@ktu.edu.tr}

\address[a]{Departamento de F\'{\i}sica y Matem\'aticas, Universidad de Alcal\'a,
E-28871 Madrid, Spain}
\address[b]{Department of Mathematics, Karadeniz Technical University, Trabzon, Türkiye}

%\fntext[proy]{Supported by the grant PID2020-113192GB-I00 (Mathematical Visualization: Foundations, Algorithms and Applications) from the Spanish MICINN. Juan G. Alc\'azar and Carlos Hermoso are also members of the Research Group {\sc asynacs} (Ref. {\sc ccee2011/r34}).  }

\begin{abstract}
Two curves are affinely equivalent if there exists an affine mapping transforming one of them onto the other. Thus, detecting affine equivalence comprises, as important particular cases, similarity, congruence and symmetry detection. In this paper we generalize previous results by the authors to provide an algorithm for computing the affine equivalences between two parametric curves of certain types, in any dimension. In more detail, the algorithm is valid for rational curves, and for parametric curves with non-rational but meromorphic components admitting a rational inverse. Unlike other algorithms already known for rational curves, the algorithm completely avoids polynomial system solving, and uses bivariate factoring, instead, as a fundamental tool. The algorithm has been implemented in the computer algebra system {\tt Maple}, and can be freely downloaded and used.  
\end{abstract}

\end{frontmatter}

\section{Introduction.}

We say that two curves are \emph{affinely equivalent} if one of them is the image of the other curve by means of an affine mapping. If the affine mapping preserves angles, then the two curves are \emph{similar}, i.e. both correspond to the same shape and differ only in position and$/$or scaling. If the affine mapping preserves distances, the curves are \emph{congruent}, so they differ only in position. Finally, if the two curves coincide, finding the self-congruences of the curve is equivalent to computing its \emph{symmetries}.

Because of the nature of the problem, it has received some attention in the applied fields of Computer Aided Geometric Design, Pattern Recognition and Computer Vision. In the last years, the problem has been also addressed in the Computer Algebra field, and this paper follows this trend. Examples of papers where this question has been studied are \cite{BLV, Gozutok, HJ18}; for other related papers, the interested reader can check the bibliographies of \cite{BLV, Gozutok, HJ18}. These papers address the problem for \emph{rational} curves, i.e. parametric curves whose components are quotients of polynomial functions, and aim to the more general question of checking projective, and not just affine, equivalence. While \cite{BLV,HJ18} provide solutions for this problem, with different strategies, for curves in any dimension and using as a fundamental tool polynomial system solving, the paper \cite{Gozutok} addresses the question only for space rational curves, but employing bivariate factoring as an alternative to solving polynomial systems; this leads to better timings and performance. To do this, in \cite{Gozutok} two \emph{rational invariants}, i.e. two functions rationally depending on the parametrizations to be studied which stay invariant under projective transformations, are found and used. 

In this paper we generalize the ideas of \cite{Gozutok} in three different ways. First, while the development of the invariants in \cite{Gozutok} was more of an ``art" than of a ``craft", here we provide a complete algorithm to generate such invariants. Second, the technique is valid for curves in any dimension, and not just space curves, which was the case addressed in \cite{Gozutok}: we provide an algorithm, that can be downloaded from \cite{website}, to generate the corresponding invariants for any dimension, that needs to be executed just once for each dimension. Third, our strategy is valid not only for rational curves, but also for non-algebraic, parametric curves with meromorphic components under certain conditions, so we can apply the algorithm for helices or 3D spirals, for instance, whenever some hypotheses are fulfilled. In order to also include this type of non-algebraic curves, we stick to affine equivalences, and not projective equivalences, although our ideas could be developed in a projective setting for rational curves. Notice that strategies like \cite{BLV, HJ18}, based on polynomial system solving, cannot handle non-algebraic curves because we would need to solve systems involving analytic functions. 

The idea in \cite{BLV, Gozutok, HJ18} is to compute the equivalences between the curves by determining first an associated function in the parameter space. If the curve parametrizations are rational and \emph{proper}, i.e. generically injective, the fact that the M\"obius functions are the birational transformations of the complex line guarantees that the associated function is a M\"obius function. Thus, one aims to compute first the associated M\"obius functions, and derives the equivalences between the curves from them. However, the M\"obius functions are also the bi-meromorphic transformations of the complex line, which opens up the possibility, which we explore in this paper, of solving the problem for more general, not necessarily rational parametrizations. However, in order to do this we need that these parametrizations have meromorphic \emph{global} inverses. For proper rational parametrizations, it is well-known that a global inverse exists, and is rational. For non-rational parametrizations, guaranteeing the existence of a global inverse is difficult. For this reason, 
in the case of non-rational curves we restrict ourselves to certain types of parametrizations: despite of the fact that these parametrizations are not rational, we can guarantee that they have a global, \emph{rational} inverse, therefore meromorphic, so our ideas can be applied. 

%As in the case of \cite{Gozutok}, the main computational tool that we need is bivariate factoring. For rational curves this functionality is implemented in any computer algebra system; in our case we used {\tt Maple}. For non-rational curves we depend on the ability of the computer algebra system to factor bivariate non-algebraic expressions. We noticed, and were surprised by the fact, that {\tt Maple} is very good at this. 

Although the implementation of our algorithm is relatively simple, and can be downloaded from \cite{website} jointly with the examples worked out in this paper, its development is involved. Because of this, we present the main ideas and results behind the algorithm, as well as several examples, in Section \ref{state} and Section \ref{firstsep}, so that the reader only interested in using the algorithm can get a full idea by the end of Section \ref{firstsep}. Justifying the correctness of the algorithm requires to justify three steps that we refer to as steps (i), (ii) and (iii): while (i) and (iii) are more accessible, and are fully justified in Section \ref{firstsep}, justifying step (ii), which is the step where we generate the invariants, requires major work and is postponed to Section \ref{difficult}. Nevertheless, the intuitive idea behind the justification of step (ii) is also given in Section \ref{firstsep}. We close the paper with our conclusion and some open questions in Section \ref{conclusion}. Some technical proofs are provided in Appendix I.

\section{Background, statement of the problem and required tools.}\label{state}

\subsection{Background, auxiliary results and statement of the problem}

Let us start by describing the kind of parametric curves we will work with. The key idea is that they must be parametric curves with rational inverses. However, we will make precise their structure so that we can algorithmically verify that this requirement is satisfied. 

We need two ingredients to do this; the first ingredient is a meromorphic function $\xi:{\Bbb C}\to {\Bbb C}$. Defining
\[
\Pi:{\Bbb C}\to {\Bbb C}^2,\mbox{ }\Pi(z)=(z,\xi(z)),
\]
we observe that $\Pi$ is an invertible function over its image, which is the graph ${\mathcal G}_\xi$ of the function $\xi$,
\[
{\mathcal G}_\xi=\{(z,\xi(z))|z\in {\Bbb C}\}\subset {\Bbb C}^2.
\]
Indeed, for $(z,\omega)\in {\Bbb C}^2$, $\omega=\xi(z)$, we have $\Pi^{-1}(z,\omega)=z$. The second ingredient is a rational mapping $\Phi:{\Bbb C}^2\to {\Bbb C}^n$. If we compose these two mappings, we get a new mapping 
\begin{equation}\label{thex}
\bfp=\Phi\circ \Pi:{\Bbb C}\to {\Bbb C}^n
\end{equation}
which provides a parametrization $\bfp(z)=\Phi(z,\xi(z))$ of a curve $\bfC\subset {\Bbb C}^n$, which is the image of ${\mathcal G}_\xi$ under $\Phi$. Notice that $\bfp$ is a vector function with meromorphic components. Of course if $\xi$ is a rational function, $\bfp$ is just a rational parametrization. We will also assume that the curve defined by $\bfp$ is not contained in a hyperplane.

What we really need is the condition that the restriction $\Pi|_{{\mathcal G}_\xi}$ is birational, so that $\bfp$ has a rational inverse. If $\xi$ is a rational function, what implies that $\bfp(z)$ is a rational parametrization, we will just assume that $\bfp(z)$ is \emph{proper}, i.e. generically injective, in which case $\bfp^{-1}$ exists and is rational. So let us provide sufficient conditions to guarantee this also in the case when $\xi$ is not a rational function. In that case, we will assume that $\Phi$ is a birational mapping. Recall that the \emph{cardinality} of the fiber of $\Phi$, which we denote by $\#(\Phi)$, is the number of points in the preimage of $\Phi({\bf q})$ with ${\bf q}\in {\Bbb C}^2$ a generic point. Then, the birationality of $\Phi$ is equivalent to $\#(\Phi)=1$ (see for instance Proposition 7.16 in \cite{Harris}), so we can check this condition by just picking a random point ${\bf q}$, and computing the number of points in the preimage of $\Phi({\bf q})$. Furthermore, we have the following lemma, inspired by \cite{PS04,PSV11}.

\begin{lemma}\label{card}
Let $\Phi:{\Bbb C}^2\to {\Bbb C}^n$ be a birational mapping. Then the set of points ${\bf q}\in {\Bbb C}^2$ such that $\#(\Phi({\bf q}))>1$ is included in an algebraic variety ${\mathcal V}\subset{\Bbb C}^2$ of dimension at most 1. 
\end{lemma}

\begin{proof} Let 
\[
\Phi(x_1,x_2)=(\Phi_1(x_1,x_2),\ldots,\Phi_n(x_1,x_n))=(y_1,\ldots,y_n).
\]
If $\Phi$ is birational, then $\Phi^{-1}$ exists and is rational, i.e.
\[
\Phi^{-1}(y_1,\ldots,y_n)=(\Psi_1(y_1,\ldots,y_n),\Psi_2(y_1,\ldots,y_n))
\]
where for $i=1,2$,
\[
\Psi_i(y_1,\ldots,y_n)=\dfrac{A_i(y_1,\ldots,y_n)}{B_i(y_1,\ldots,y_n)}
\]
with $A_i,B_i$ polynomials. Then the set of points ${\bf q}\in {\Bbb C}^2$ such that $\#(\Phi({\bf q}))>1$ is included in the set ${\mathcal V}\subset {\Bbb C}^2$ where $\Phi^{-1}$ is not defined, ${\mathcal V}$ being the union of the sets defined by $(B_i\circ \Phi)(x_1,x_2)=0$ with $i=1,2$, and $N_i(x_1,x_2)=0$, $i=1,2$, where $N_i$ is the denominator of $(A_i\circ \Phi)(x_1,x_2)$. Notice that ${\mathcal V}$ is an algebraic planar curve. 
\end{proof}

\begin{corollary}\label{prop}
Assume that $\Phi$ is a birational mapping, and $\xi$ is a meromorphic, not algebraic function. Then $\bfp(z)$ is invertible over $\bfC=\Phi({\mathcal G}_\xi)$, and the inverse $\bfp^{-1}$ has rational components.
\end{corollary}

\begin{proof} Since by assumption $\#(\Phi)=1$, $\Phi^{-1}$ exists and is rational. Next from Lemma \ref{card} we have that $\#(\Phi)$ is constant except perhaps for the points of an algebraic variety ${\mathcal V}\subset {\Bbb C}^2$ of dimension at most one. Since $\xi$ is not an algebraic function ${\mathcal G}_\xi$ is not an algebraic curve. Therefore, by the Identity Theorem (see Theorem 3.1.9 in \cite{Jong}) ${\mathcal G}_\xi\cap {\mathcal V}$ is either finite, or infinite but without any accumulation point. Thus, for a generic point ${\bf q}\in {\mathcal G}_\xi$ we get that the cardinality of $\Phi({\bf q})$ is $1$. Therefore $\Phi^{-1}$ is well-defined for almost all points in $\Phi({\mathcal G}_\xi)$, and $\bfp^{-1}=\Pi^{-1}\circ\Phi^{-1}$. Since $\Phi^{-1},\Pi^{-1}$ are rational, $\bfp^{-1}$ is rational as well. 
\end{proof}

\begin{example}\label{somecurves}
To illustrate the assumptions that we need, we provide now some examples of plane and space curves satisfying these requirements. 
\begin{itemize}
\item [(1)] {\it Catenary.} Consider the curve $\bfC$ parametrized by 
\[
(z,\mbox{cosh}(z)),
\]
where $\mbox{cosh}$ denotes the hyperbolic cosine (see Fig. \ref{analytic}, left). Here $\xi(z)=\mbox{cosh}(z)$ and $\Phi(x,y)=(x,y)$, which is clearly birational. 
\item [(2)] {\it Image of the graph of the exponential curve under an inversion.} Let the curve $\bfC$ (see Fig. \ref{analytic}, middle) be parametrized by 
\[
\left(\dfrac{z}{z^2+e^{2z}},\dfrac{e^z}{z^2+e^{2z}}\right)
\]
Here $\xi(z)=e^z$ and 
\[
\Phi(x,y)=\left(\dfrac{x}{x^2+y^2},\dfrac{y}{x^2+y^2}\right),
\]
which is an inversion from the origin, and therefore a birational mapping. 
\item [(3)] {\it 3D spiral.} Consider the curve $\bfC$ parametrized by $(z\cos(z),z\sin(z),z)$, which is a 3D spiral (see Fig. \ref{analytic}, right). Writing 
\[
\cos(z)=\dfrac{e^{2{\bf i}z}+1}{2e^{{\bf i}z}},\mbox{ }\sin(z)=\dfrac{e^{2{\bf i}z}-1}{2{\bf i}e^{{\bf i}z}},\mbox{ }{\bf i}^2=-1,
\]
the parametrization of $\bfC$ can be expressed as 
\[
\left(z\dfrac{e^{2{\bf i}z}+1}{2e^{{\bf i}z}},z\dfrac{e^{2{\bf i}z}-1}{2{\bf i}e^{{\bf i}z}},z\right).
\]
Then, $\xi(z)=e^{{\bf i}z}$, and 
\[
\Phi(x,y)=\left(y\dfrac{x^2+1}{2x},y\dfrac{x^2-1}{2{\bf i}x},y\right),
\]
which is a birational mapping.
\end{itemize} 
\end{example}

\begin{figure}
\begin{center}
\centerline{$\begin{array}{ccc}
\includegraphics[scale=0.35]{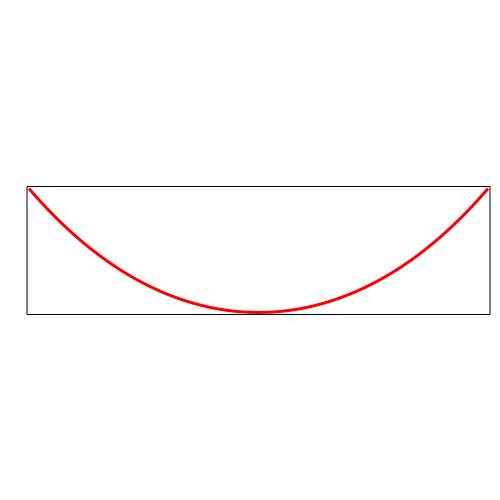} &\includegraphics[scale=0.25]{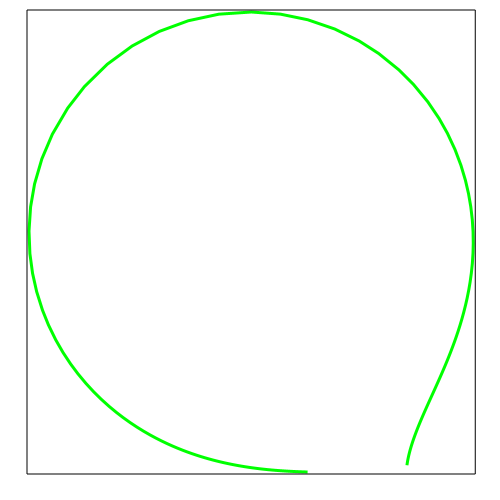} &
\includegraphics[scale=0.25]{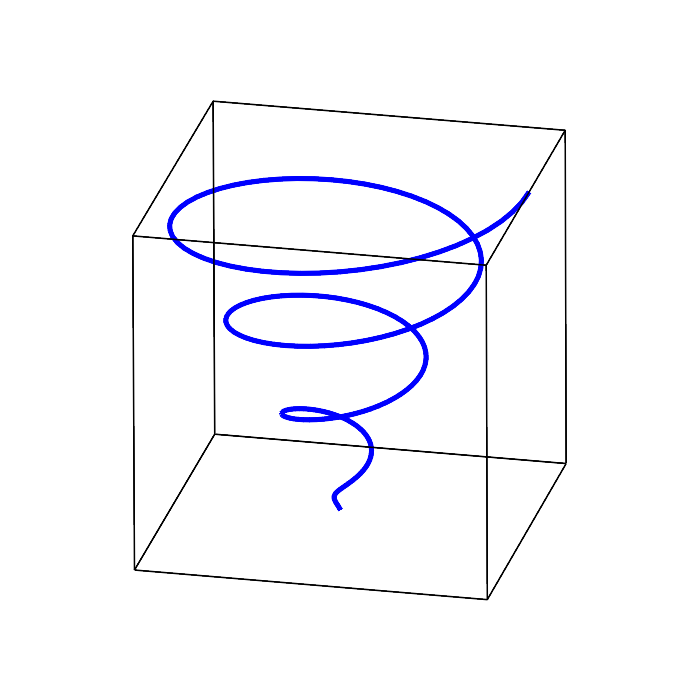} 
\end{array}$}
\end{center}
\caption{Some curves with meromorphic parametrizations, and rational inverses}\label{analytic}
\end{figure}

Even though for technical reasons we will assume that the parameter space of the curve $\bfC$ parametrized by Eq. \eqref{thex} is ${\Bbb C}$, we will mostly work with real curves, i.e. curves with infinitely many real points. Next, given two curves $\bfC_1,\bfC_2\subset{\Bbb C}^n$ parametrized by $\bfp(z),\bfq(z)$ as in Eq. \eqref{thex}, we say that $\bfC_1,\bfC_2$ are \emph{affinely equivalent} if there exists a mapping $f:{\Bbb C}^n\to {\Bbb C}^n$, $f(\bfx)=A\bfx+\bfb$ with $A\in {\mathcal M}_{n\times n}({\Bbb C})$, i.e. $A$ is an $n\times n$ matrix (in general, over the complex), $A$ non-singular, $\bfb\in {\Bbb C}^n$, such that $f(\bfC_1)=\bfC_2$; furthermore, we say that $f$ is an \emph{affine equivalence} between $\bfC_1,\bfC_2$. We are interested in real affine equivalences, so in our case we will be mostly looking for $A\in {\mathcal M}_{n\times n}({\Bbb R})$, $\bfb\in {\Bbb R}^n$. If $\bfC_1=\bfC_2=\bfC$ and $A$ is \emph{orthogonal}, i.e. $A^TA=I$, with $I$ the identity matrix, we say that $f$ is a \emph{symmetry} of $\bfC$. 

\vspace{0.3 cm}
Now we are ready to state the problem that we want to solve. 

\vspace{0.3 cm}
\noindent
\underline{\bf Problem:} Given two curves $\bfC_1,\bfC_2\subset {\Bbb C}^n$, not contained in hyperplanes, parametrized by mappings $\bfp(z),\bfq(z)$ as in Eq. \eqref{thex} with meromorphic components, admitting rational inverses $\bfp^{-1},\bfq^{-1}$, compute the affine equivalences, if any, between $\bfC_1,\bfC_2$. 

\vspace{0.3 cm}
In order to solve this problem, we will make use of the following result, which corresponds to a similar result used in \cite{Gozutok,HJ18}, adapted to our case. We recall here that a \emph{M\"obius transformation} is a transformation
\begin{equation*}
\varphi:{\Bbb C}\to{\Bbb C},\quad \varphi(z)=\dfrac{az+b}{cz+d},\quad ad-bc\neq 0.
\end{equation*}

\begin{theorem}\label{th-ref}
Let $\bfC_1,\bfC_2\subset {\Bbb C}^n$ be two parametric curves defined by $\bfp(z),\bfq(z)$ as in Eq. \eqref{thex}, where $\bfp(z),\bfq(z)$ are mappings with meromorphic components, admitting rational inverses $\bfp^{-1},\bfq^{-1}$. If $f(\bfx)=A\bfx+\bfb$ is an affine equivalence between $\bfC_1,\bfC_2$ then there exists a M\"obius transformation $\varphi(z)$ satisfying that 
\begin{equation}\label{fund}
f\circ \bfp=\bfq\circ \varphi,
\end{equation}
i.e. making commutative the following diagram:
\begin{equation}\label{eq:fundamentaldiagram}
\xymatrix{
\bfC_1 \ar[r]^{f} & \bfC_2 \\
{\Bbb C} \ar@{-->}[u]^{\bfp} \ar@{-->}[r]_{\varphi} & {\Bbb C} \ar@{-->}[u]_{\bfq}
}
\end{equation}
\end{theorem}

\begin{proof} Since $\bfq^{-1}$ exists, $\varphi=\bfq^{-1}\circ f \circ \bfq$ is well-defined. Furthermore, since $\bfq^{-1}$ is rational, $\bfq^{-1}\circ f$ is also a rational function and therefore $\varphi=(\bfq^{-1}\circ f)\circ \bfp$ is meromorphic. Since $\bfp$ exists and is rational, $\varphi^{-1}=\bfp^{-1}\circ f^{-1}\circ \bfq$ is also meromorphic, so $\varphi$ is a bi-meromorphic function, and therefore it must be a M\"obius transformation (see Remark 2 in \cite{AM22}).
\end{proof} 

\begin{remark} 
Theorem \ref{th-ref} also works with meromorphic parametrizations $\bfp,\bfq$ admitting \emph{global} meromorphic inverses. However, guaranteeing the existence of a global meromorphic inverse is a really hard problem. This is the reason why we restrict ourselves to 
parametrizations, rational or not rational, where this condition is easy to check. Notice that the parametrizations we work with here have rational inverses, so certainly they have global meromorphic inverses. 
\end{remark}

\subsection{Additional tools} \label{addtools}

In this subsection we recall two notions that we will be using later in the paper. The first one is the \emph{Schwartzian derivative}: given a holomorphic function $f:{\Bbb C}\to {\Bbb C}$, the Schwartzian derivative \cite{OT09} $S(f)$ of $f$ is 
\[
S(f)(z)=\dfrac{f'''(z)}{f'(z)}-\dfrac{3}{2}\left(\dfrac{f''(z)}{f'(z)}\right)^2.
\]
The Schwartzian derivative of any M\"obius transformation is identically zero. The following lemma is a consequence of this. 

\begin{lemma}\label{schwa}
Let $\omega:=\varphi(z)$ a M\"obius transformation, and let $\omega^{(k)}$ denote the $k$-th derivative of $\omega$ with respect to $k$. For $k\geq 3$, 
\begin{equation}\label{schwa3}
\omega^{(k)}=\dfrac{k!}{2^{k-1}}\dfrac{(\omega'')^{k-1}}{(\omega')^{k-2}}.
\end{equation}
\end{lemma}

\begin{proof} Since the Schwartzian derivative of a M\"obius transformation is identically zero, we get that 
\[
\omega'''=\dfrac{3}{2}\dfrac{(\omega'')^2}{\omega'},
\]
which corresponds to Eq. \eqref{schwa3} for $k=3$. Then the result follows by induction on $k$. 
\end{proof}

The second tool that we will need is \emph{Fa\`a di Bruno's formula} \cite{Jonson} for the derivatives of high order of a composite function. Given a vector function $\bfu:=\bfu(z)$ and a scalar function $\omega:=\omega(z)$, Fa\`a di Bruno's formula provides the derivatives of order $k\geq 1$ of the composite function $\bfu(\omega)$ with respect to $z$. Although there are other formulations, we will use the expression 
\begin{equation}\label{bell}
\dfrac{d^k(\bfu(\omega))}{dz^k}=\sum_{m=1}^k \bfu^{(m)}(\omega)B_{k,m}(\omega',\omega'',\ldots,\omega^{(k+1-m)}),
\end{equation}
where the $B_{k,m}$ are the incomplete (or partial) \emph{Bell polynomials} \cite{K21}, well-known in combinatorics,
\begin{equation}\label{bell2}
B_{k,m}(x_1,x_2,\ldots,x_{k+1-m})=\sum \dfrac{k!}{\ell_1! \cdots \ell_{k+1-m}!}\prod_{i=1}^{k+1-m} \left(\dfrac{x_i}{i!}\right)^{\ell_i},
\end{equation}
where the sum is taken over all sequences $\ell_1,\ell_2,\ldots,\ell_{k+1-m}$ of non-negative integers such that 
\begin{equation}\label{someq1}
\sum_{i=1}^{k+1-m}\ell_i=m,\mbox{ }\sum_{i=1}^{k+1-m}i\ell_i=k.
\end{equation}
In particular, 
\begin{equation}\label{someq2}
B_{k,k}(x_1)=(x_1)^k,
\end{equation}
and we will assume the convention that $B_{k,m}=0$ when $k<m$.

%The formula in Eq. \eqref{bell} can be easily computed, for instance, with the help of the command {\tt IncompleteBellB} in the computer algebra Maple. For instance, 
%\begin{equation*}
%\begin{split}
%\dfrac{d^3(\bfu(\omega))}{dz^3}&=\sum_{m=1}^{3}{\bfu^{(m)}(\omega)B_{3,m}(\omega',\ldots,\omega^{(4-m)})}=\\
%&=\bfu'(\omega)B_{3,1}(\omega',\omega'',\omega''')+\bfu''(\omega)B_{3,2}(\omega',\omega'')+\bfu'''(s)B_{3,3}(\omega')=\\
%&=\omega'''\bfu'(\omega)+3\omega'\omega''\bfu''(\omega)+s'^3\bfu'''(s).
%\end{split}
%\end{equation*}

\section{Development of the method, algorithm and examples}\label{firstsep}

\subsection{Overall strategy and first step}\label{overall}

We want to exploit Eq. \eqref{fund} to first find the M\"obius transformation $\varphi$, if any, and then derive $f$ from $\varphi$. If we expand Eq. \eqref{fund}, we get
\begin{equation}\label{fund2}
A\bfp(z)+\bfb=(\bfq\circ \varphi)(z).
\end{equation}
Our overall strategy will consist of three steps, that we will refer to as steps (i), (ii), (iii):

\begin{itemize}
\item [(i)] {\it Find initial invariants:} we start by constructing certain functions $I_1,\ldots,I_n$ satisfying that $I_i(\bfp)=I_i(\bfq\circ \varphi)$, which are \emph{rational} in the sense that they are rational functions of $\bfp$ and its derivatives. Since by Eq. \eqref{fund2} we observe that $\bfq\circ \varphi$ is the image of $\bfp$ under an affine mapping $f(\bfx)=A\bfx+\bfb$, we say that $I_1,\ldots,I_n$ are \emph{affine invariants}, i.e. functions depending on a parametrization (and its derivatives) that stay the same when an affine transformation is applied. 

\item [(ii)] {\it Find M\"obius-commuting invariants:} we say that a function $F$ depending on a parametrization $\bfu=\bfu(z)$ and its derivatives is \emph{M\"obius-commuting} if for any M\"obius function we have 
\[F(\bfu\circ \varphi)=F(\bfu)\circ \varphi.\]
The functions $I_i$ found in step (i) are not, in general, M\"obius-commuting. Thus, in a second step we will compute M\"obius-commuting functions $F_1,\ldots,F_{n-1}$ from the $I_i$, also rational. The $F_j$ not only satify that $F_j(\bfp)=F_j(\bfq \circ \varphi)$ for $j=1,\ldots,n-1$, but they also satisfy that $F_j(\bfq\circ \varphi)=F_j(\bfq)\circ \varphi$. In turn, for $j=1,\ldots,n-1$ we have 
\[
F_j(\bfp)=F_j(\bfq)\circ \varphi.
\]
Notice that while we have $n$ initial invariants $I_i$, we have $n-1$ M\"obius-commuting invariants. 

\item [(iii)] {\it Compute $\varphi$ using bivariate factoring, and derive $f$ from $\varphi$:} setting $\omega:=\varphi(z)$, the equalities $F_j(\bfp)=F_j(\bfq)\circ \varphi$, after clearing denominators, are translated into $n-1$ conditions $M_j(z,\omega)=0$, with $j=1,\ldots,n-1$. Then the M\"obius function $\varphi$ corresponds to a common factor of all the $M_j$, and the affine equivalence itself, $f(\bfx)=A\bfx+\bfb$, follows from Eq. \eqref{fund2}. 
\end{itemize}

In this subsection we will present step (i); the remaining steps will be described in the next section. Also, in the rest of the paper we will use the notation $[\bfw_1,\cdots,\bfw_n]$ for an $n\times n$ matrix whose columns are $\bfw_1,\ldots,\bfw_n\in {\Bbb C}^n$, and $\Vert \bfw_1,\cdots,\bfw_n\Vert$ for the determinant of the matrix $[\bfw_1,\cdots,\bfw_n]$. 

\vspace{0.3 cm}
The description of step (i) is analogous to Section 3.2 in \cite{Gozutok}. Thus, here we focus on the main ideas, and refer the interested reader to \cite{Gozutok} for details and proofs. Going back to Eq. \eqref{fund}, let us write $\bfu:=\bfp(z)$, $\bfv:=(\bfq\circ \varphi)(z)$, so that Eq. \eqref{fund2} becomes simply $A\bfu+\bfb=\bfv$. Repeatedly differentiating this equation with respect to $z$ yields $AD(\bfu)=D(\bfv)$ where 
\[
D(\bfu)=\left[\bfu', \bfu'',\cdots ,\bfu^{(n)}\right],\mbox{ }D(\bfv)=\left[\bfv', \bfv'',\cdots ,\bfv^{(n)}\right],
\]
i.e. $D(\bfu),D(\bfv)$ are matrices whose columns consist of the first $n$ derivatives of $\bfu,\bfv$ with respect to $z$. Whenever $\bfp,\bfq$ and therefore $\bfu,\bfv$ are not contained in hyperplanes, $D(\bfu),D(\bfv)$ are invertible \cite{Sulanke}. Thus, we can write $A=D(\bfv)(D(\bfu))^{-1}$. Differentiating this equality with respect to $z$, and taking into account that $A$ is a constant matrix, we get that \[\dfrac{d(D(\bfv)(D(\bfu))^{-1})}{dz}=0.\]Expanding the derivative in the left-hand side of the above equation we arrive at
\begin{equation}\label{UVequal}
(D(\bfu))^{-1}\dfrac{dD(\bfu)}{dz}=(D(\bfv))^{-1}\dfrac{dD(\bfv)}{dz}.
\end{equation}

Denoting
\[
U=(D(\bfu))^{-1}\dfrac{dD(u)}{dz},\mbox{ }V=(D(\bfv))^{-1}\dfrac{dD(v)}{dz},
\]
one can check that 
\begin{equation}\label{UV}
\begin{array}{cc}
U=\begin{bmatrix}
0 & 0 & \cdots & 0 & \dfrac{\Vert \bfu^{(n+1)},\bfu'',\cdots,\bfu^{(n)}\Vert}{\Vert\bfu',\bfu'',\cdots,\bfu^{(n)}\Vert} \\
0 & 1 & \cdots & 0 & \dfrac{\Vert \bfu',\bfu^{(n+1)},\cdots,\bfu^{(n)}\Vert}{\Vert\bfu',\bfu'',\cdots,\bfu^{(n)}\Vert}\\
\vdots & \vdots &\ddots & \vdots & \vdots\\
0 & 0 & \cdots & 1 & \dfrac{\Vert \bfu',\bfu'',\cdots,\bfu^{(n+1)}\Vert}{\Vert\bfu',\bfu'',\cdots,\bfu^{(n)}\Vert}
\end{bmatrix},
V=\begin{bmatrix}
0 & 0 & \cdots & 0 & \dfrac{\Vert \bfv^{(n+1)},\bfv'',\cdots,\bfv^{(n)}\Vert}{\Vert\bfv',\bfv'',\cdots,\bfv^{(n)}\Vert} \\
0 & 1 & \cdots & 0 & \dfrac{\Vert \bfv',\bfv^{(n+1)},\cdots,\bfv^{(n)}\Vert}{\Vert\bfv',\bfv'',\cdots,\bfv^{(n)}\Vert}\\
\vdots & \vdots &\ddots & \vdots & \vdots\\
0 & 0 & \cdots & 1 & \dfrac{\Vert \bfv',\bfv'',\cdots,\bfv^{(n+1)}\Vert}{\Vert\bfv',\bfv'',\cdots,\bfv^{(n)}\Vert}
\end{bmatrix}.
\end{array}
\end{equation}

Next let us define 
\begin{equation}\label{As}
A_i(\bfu):=\Vert \bfu',\cdots,\bfu^{(i-1)},\bfu^{(n+1)},\bfu^{i+1},\cdots, \bfu^{(n)}\Vert, \mbox{ }\Delta(\bm{u}):=\Vert \bm{u}',\bm{u}'',\cdots,\bfu^{n}\Vert.
\end{equation}
Thus, $A_i(\bfu)$ is the result of replacing $\bfu^{(i)}$ in $\Vert \bfu' \bfu'' \cdots \bfu^{(n)}\Vert$ by $\bfu^{(n+1)}$. Finally, for $i=1,\ldots,n$, let
\begin{equation}\label{first-inv}
I_i(\bm{u}):=\dfrac{A_i(\bm{u})}{\Delta(\bm{u})},
\end{equation}
which correspond to the entries of the last column of $U$; notice that whenever $\bfu,\bfv$ are not contained in hyperplanes $\Delta(\bfu)$ is not identically zero \cite{Sulanke}, so the $I_i$ are well defined. 

\vspace{0.3 cm}
By Eq. \eqref{UVequal}, $U,V$ are equal and therefore their last columns coincide. Thus, $I_i(\bfu)=I_i(\bfv)$ for $i=1,\ldots,n$, i.e. $I_i(\bfp)=I_i(\bfq\circ \varphi)$, which by Theorem \ref{th-ref} is a necessary condition for affine equivalence. The following result, analogous to Theorem 7 in \cite{Gozutok} and which can be proved, using Theorem \ref{th-ref}, in a similar way, shows that this condition is also sufficient.

\begin{theorem}\label{teo29*}
Let $\bm{C}_1,\bm{C}_2\subset {\Bbb C}^n$ be two curves, not contained in a hyperplane, parametrized by mappings $\bm{p},\bm{q}$ with meromorphic components, admitting rational inverses. If $\bm{C}_1,\bm{C}_2$ are affinely equivalent then there exists a M\"obius transformation $\varphi$ such that 
\begin{equation}\label{Ipqmob}
I_i(\bfp)=I_i(\bfq\circ \varphi)
\end{equation}
for $i=1,\ldots,n$. 
\end{theorem}

In order to carry out step (ii), which will be addressed in the next section, we need an auxiliary invariant, $I_0$, defined as
\begin{equation}\label{I3}
I_{0}(\bfu):=\dfrac{\Vert \bfu',\cdots,\bfu^{(n-1)},\bfu^{(n+2)}\Vert}{\Vert \bfu',\bfu'',\cdots,\bfu^{(n)}\Vert}.
\end{equation} 

The following lemma proves that $I_{0}$ lies in the differential field spanned by $I_1,\ldots,I_n$. The proof of this lemma is provided in Appendix I. 

\begin{lemma}\label{lemmaI0}
$I_{0}=\dfrac{dI_n}{dz}+I_{n-1}+I_n^2$.  
\end{lemma}

Since, according to Lemma \ref{lemmaI0}, $I_0$ is generated by $I_1,\ldots,I_n$, the result in Theorem \ref{teo29*} also holds when we add $I_0$ to the list of the $I_i$s. 

\begin{corollary}\label{teo29**}
Let $\bm{C}_1,\bm{C}_2\subset {\Bbb C}^n$ be two curves, not contained in a hyperplane, parametrized by mappings $\bm{p},\bm{q}$ with meromorphic components, admitting rational inverses. If $\bm{C}_1,\bm{C}_2$ are affinely equivalent then there exists a M\"obius transformation $\varphi$ such that 
\begin{equation}\label{Ipqmob*}
I_i(\bfp)=I_i(\bfq\circ \varphi)
\end{equation}
for $i\in\{0,1\ldots,n\}$. 
\end{corollary} 

\subsection{Second step (overview) and third step}\label{secondstep}

The $I_i$ developed in the previous section are not M\"obius-commuting, i.e. $I_i(\bfq\circ\varphi)\neq I_i(\bfq)\circ \varphi$; in other words, calling $\omega:=\varphi(z)$, $I_i(\bfq(\omega))\neq I_i(\bfq)(\omega)$. For instance, in the case $n=3$, expanding $I_i(\bfq(\omega))$ for $i=1,2,3$ we get that
\begin{equation}\label{exsec2}
\begin{array}{rcl}
\omega'^3I_1(\bm{q}(\omega)) &=&3\omega''^3+\dfrac{3}{2}\omega'^2\omega''^2I_3(\bfq)(\omega)-\omega'^4\omega''I_2(\bfq)(\omega)+\omega'^6I_1(\bfq)(\omega) \\
\omega'^2I_2(\bm{q}(\omega)) &=&-9\omega''^2+\omega'^4I_2(\bfq)(\omega)-3\omega'^2\omega''I_3(\bfq)(\omega) \\ 
\omega'I_3(\bm{q}(\omega)) &=&6\omega''+\omega'^2I_3(\bfq)(\omega),\\
\end{array}
\end{equation}
where $\omega',\omega''$ are the first and second derivatives of $\omega=\varphi(z)$ with respect to $z$; to produce these equalities, we have taken into account the definition of $I_1,I_2,I_3$ as quotients of determinants, the Chain Rule, and the fact that, because of Eq. \eqref{schwa3} in Lemma \ref{schwa}, the derivatives of $\omega$ of order higher than 3 can be written in terms of $\omega',\omega''$. However, by eliminating $\omega',\omega''$ in Eq. \eqref{exsec2}, one can show that

\begin{equation}\label{exsec22}
\dfrac{\left[36I_1(\bm{q}(\omega))+6I_2(\bm{q}(\omega))I_3(\bm{q}(\omega))+I_3(\bm{q}(\omega))^3\right]^2}{\left[4I_2(\bm{q}(\omega))+I_3(\bm{q}(\omega))^2\right]^3}=\dfrac{\left[36I_1(\bfq)(\omega)+6I_2(\bfq)(\omega)I_3(\bfq)(\omega)+I_3(\bfq)(\omega)^3\right]^2}{\left[4I_2(\bfq)(\omega)+I_3(\bfq)(\omega)^2\right]^3},
\end{equation}
so that 
\begin{equation}\label{newinv}
F=\dfrac{\left[36I_1(\bfq)+6I_2(\bfq)I_3(\bfq)+I_3^3(\bfq)\right]^2}{\left[4I_2(\bfq)+I_3^2(\bfq)\right]^3}
\end{equation}
is M\"obius-commuting, i.e. $F(\bfq\circ \varphi)=F(\bfq)\circ \varphi$. 

\vspace{0.3 cm}
One can certainly manipulate Eq. \eqref{exsec2} by hand to get rid of $\omega',\omega''$, reach Eq. \eqref{exsec22}, and therefore find the invariant in Eq. \eqref{newinv}. However, we want to produce invariants like the one in Eq. \eqref{newinv} in an algorithmic fashion, and for any dimension: that is the task in step (ii). The rough idea, as in Eq. \eqref{exsec2}, is to get rid of the derivatives $\omega^{(k)}$, $k=1,2,\ldots,n+2$, in the system consisting of the expressions
\[I_i(\bfq(\omega))=\xi_i\left(I_0(\bfq)(\omega),\ldots,I_n(\bfq)(\omega),\omega',\omega'',\ldots,\omega^{(n+2)}\right),\]
where $\xi_i$ is the result of expanding $I_i(\bfq(\omega))$, with $i=0,1,\ldots,n$. The process is involved, and will be detailed in Section \ref{difficult}, but as a final product of this process we get closed expressions for these invariants (see Theorem \ref{egregium} in the next Section \ref{difficult}), that we denote $F_1,\ldots,F_{n-1}$. The generation of the M\"obius-commuting invariants, for any dimension $n$, is implemented in \cite{website}, which can be freely downloaded, and can be done just once for each dimension $n$. In Table \ref{t1} we spell the invariants for low dimension, $2\leq n\leq 4$.

\begin{table}[H]
\begin{tabular}{c|lll}
$n$ & M\"obius-commuting invariants\\
\hline\\
$2$ & $F_1=\dfrac{I_0-I_2^2}{6I_1+I_2^2}$ & & \\
$3$ & $F_1=\dfrac{6I_0-5I_3^2}{4I_2+I_3^2}$& \hspace{-1.5 cm}$F_2=\dfrac{\left(36I_1+6I_2I_3+I_3^3\right)^2}{\left(4I_2+I_3^2\right)^3}$\\
$4$ & $F_1=\dfrac{4I_0-3I_4^2}{10I_3+3I_4^2}$& \hspace{-1.5 cm}$F_2=\dfrac{\left(50I_2+15I_3I_4+3I_4^3\right)^2}{\left(10I_3+3I_4^2\right)^3}$& $F_3=\dfrac{4000I_1+400I_2I_4+60I_3I_4^2+9I_4^4}{\left(10I_3+3I_4^2\right)^2}$
\end{tabular}
\caption{M\"obius-commuting invariants for low dimension}\label{t1}
\end{table}

Next let us address step (iii). Let $F_j$ be a M\"obius-commuting invariant, $j\in \{1,\ldots,n-1\}$. Since $F_j$ is a rational function of the $I_i$, $F_j$ is also an affine invariant, i.e. from Theorem \ref{teo29*} we get that $F_j(\bfp)=F_j(\bfq\circ \varphi)$. Therefore, in terms of the variables $z$ and $\omega:=\varphi(z)$, and taking into account that $F_j(\bfq\circ \varphi)=F_j(\bfq)\circ \varphi$, we deduce that $F_j(\bfp)(z)=F_j(\bfq)(\omega)$. Then we have the following result. 

\begin{proposition}\label{mainTeo}
Let $\bm{C}_1,\bm{C}_2\subset {\Bbb C}^n$ be two curves, not contained in a hyperplane, parametrized by mappings $\bm{p},\bm{q}$ with meromorphic components, admitting rational inverses. Then $\bm{C}_1,\bm{C}_2$ are affinely equivalent if and only if there exists a M\"obius transformation $\varphi$ such that 
\begin{equation}\label{maincurve}
F_j(\bfp)(z)-F_j(\bfq)(\omega)=0,
\end{equation}
for $j\in\{1,2,\ldots,n-1\}$ with $\omega=\varphi(z)$, such that $D(\bfq \circ \varphi)(D(\bfp))^{-1}(z)$ is a constant matrix $A$ and $\bfb=(\bfq\circ\varphi-A\bfp)(z)$ is a constant vector. Furthermore, $f(\bfx)=A\bfx+\bfb$ is an affine equivalence between $\bfC_1, \bfC_2$.
\end{proposition}

\begin{proof}
$(\Rightarrow)$ Let $f$ be an affine equivalence between $\bm{C}_1,\bm{C}_2$. By Theorem \ref{th-ref}, there exists a M\"obius function $\varphi$ such that $f\circ \bfp=\bfq \circ \varphi$. By Corollary \ref{teo29**} we have that $I_i(\bfp)(z)=I_i(\bfq(\omega))$ for all $i\in\{0,\ldots,n\}$. Since the $F_j$ are rational functions of the $I_i$, $I_i(\bfp)(z)=I_i(\bfq(\omega))$ yields $F_j(\bfp)(z)=F_i(\bfq)(\omega)$ for $j\in\{1,2,\ldots,n-1\}$. Finally, writing $f(\bfx)=A\bfx+\bfb$, the condition $f\circ \bfp=\bfq \circ \varphi$ implies that $A\bfp(z)+\bfb=\bfq(\varphi(z))$, so $\bfb=(\bfq\circ\varphi-A\bfp)(z)$, which is a constant vector. Furthermore, by differentiating the condition $A\bfp(z)+\bfb=\bfq(\varphi(z))$ (see Subsection \ref{overall}) we deduce that $A=D(\bfq \circ \varphi)(D(\bfp))^{-1}(z)$. $(\Leftarrow)$ Let $\varphi$ be a Möbius transformation satisfying $F_i(\bfp(z))-F_i(\bfq)(\omega)=0$ for $\omega=\varphi(z)$. If $A=D(\bfq \circ \varphi)(D(\bfp))^{-1}(z)$ is a constant matrix, then $D(A\bfp)(z)=D(\bfq\circ \varphi)(z)$, so $A\bfp(z)-(\bfq\circ \varphi)(z)$ is a constant, equal to $-\bfb$. Therefore, $A\bfp(z)+\bfb=\bfq(\varphi(z))$. But this equality implies that $A\bfp(z)+\bfb$, which is the image of $\bm{C}_1$ under the affine mapping $f(\bfx)=A\bfx+\bfb$, and $\bfq(z)$, parametrize the same curve, namely $\bm{C}_2$. Thus, $f(\bfx)=A\bfx+\bfb$ is an affine equivalence between $\bfC_1$ and $\bfC_2$.
\end{proof}

\subsection{Algorithm and examples}\label{algexamples}

To finally turn Proposition \ref{mainTeo} into an algorithm, let $M_j(z,\omega)$ be obtained by clearing denominators in $F_j(\bfp)(z)-F_j(\bfq)(\omega)$. We need to request that $M_j(z,\omega)$ is not identically zero, which amounts to requiring that not all the $F_j$ are constant: this can happen, and an example is the circular helix $\bfp(z)=(\mbox{cos}(z),\mbox{sin}(z),z)$. If $M_j(z,\omega)$ is not zero, then $M_j(z,\omega)=0$ defines an analytic curve in the plane $z,\omega$. Now if
\begin{equation}\label{again}
\varphi(z)=\dfrac{az+b}{cz+d}
\end{equation}
is a M\"obius function satisfying Proposition \ref{mainTeo}, calling $\omega=\varphi(z)$ we get that all the points $(z,\omega)$ of the curve 
\[
\omega(cz+d)-(az+d)=0,
\]
which is an irreducible analytic curve, are also points of the curve $M_j(z,\omega)$. As a consequence of Study's Lemma (see Section 6.13 of\cite{Fischer}), $H(z,\omega)=\omega(cz+d)-(az+d)$ must be a factor of $M_j(z,\omega)$; we say that $H(z,\omega)=\omega(cz+d)-(az+d)$ is a \emph{M\"obius-like} factor of $M_j(z,\omega)$, and that the M\"obius function $\varphi$ in Eq. \eqref{again} is \emph{associated} with $H(z,\omega)$. So we have the following theorem, which follows from Proposition \ref{mainTeo}. 

\begin{theorem}\label{finalTeo}
Let $\bm{C}_1,\bm{C}_2\subset {\Bbb C}^n$ be two curves, not contained in a hyperplane, parametrized by mappings $\bm{p},\bm{q}$ with meromorphic components, admitting rational inverses and where not all the $F_j$ are constant. Then $\bm{C}_1,\bm{C}_2$ are affinely equivalent if and only if there exists a M\"obius-like factor $H(z,\omega)$ common to $M_j(z,\omega)$, $j=1,\ldots,n-1$ such that the corresponding associated M\"obius function $\varphi$ satisfies that: (1) $D(\bfq \circ \varphi)(D(\bfp))^{-1}(z)$ is a constant matrix $A$, (2) $\bfb=(\bfq\circ\varphi-A\bfp)(z)$ is a constant vector. Furthermore, in that case $f(\bfx)=A\bfx+\bfb$ is an affine equivalence between $\bfC_1, \bfC_2$.
\end{theorem}

Thus, we get the following procedure {\tt AffineEquivalences} to find the affine equivalences between the curves $\bm{C}_1,\bm{C}_2$ defined by $\bfp,\bfq$.

\begin{algorithm}[H]
\caption*{{\tt AffineEquivalences}}
\textbf{Input:} \textit{Two parametrizations $\bm{p}$ and $\bm{q}$ with meromorphic components, admitting rational inverses.} \\
\textbf{Output:} \textit{Either the list of affine equivalences between the curves, or the warning {\tt The curves are not affinely equivalent}}
\begin{algorithmic}[1]
    \Procedure{\tt AffEq}{$\bm{p},\bm{q}$}
    \State Compute $M_j(x,z)$, $j=1,\ldots,n-1$, by clearing denominators in $F_j(\bfp)(z)-F_j(\bfq)(\omega)$.
		\If{all the $M_j$ are identically zero}
    	\State {\bf return} {\tt Failure: all the M\"obius-commuting invariants are constant}
    	\Else 
    \State Compute the common factor $L(x,z)$ of the $M_j(x,z)$.
    \State Let $\mathcal{L}$ be the list of M\"obius-like factors of $L(x,z)$ 
    \If{$\mathcal{L}=\emptyset$}
    	\State {\bf return} {\tt The curves are not affinely equivalent}
    	\Else 
			\For{$\varphi\in \mathcal{L}$}
			\State Check whether or not $A=D(\bm{q}(\varphi))D(\bm{p})^{-1}$, $\bfb=\bfq\circ\varphi-A\bfp$ are constant
			\State In the affirmative case, {\bf return} $f(\bfx)=A\bfx+\bfb$.
    	\EndFor
    \EndIf
		    \EndIf
   \EndProcedure
\end{algorithmic}
\end{algorithm}

If $\bfp,\bfq$ are rational, the $M_j(x,z)$ are rational and $H(z,\omega)$ is a factor of $\mbox{gcd}(M_1(z,\omega),\ldots,M_{n-1}(z,\omega))$. However, the computer algebra system {\tt Maple}, where we implemented the procedure (see \cite{website}), can compute $H(z,\omega)$ also in the case when $\bfp,\bfq$ are not rational, but satisfy the hypotheses of the procedure. In this last case, we ask {\tt Maple} to solve $H(z,\omega)$ for $\omega$ to find the M\"obius functions. 

\begin{remark}
Although {\tt Maple Help System} is not too specific about this, in the case when the $M_j(x,z)$ are not rational the idea seems to be that {\tt Maple} renames repeated non-rational expressions found in the $M_j(x,z)$ (e.g. $\mbox{cos}(z),e^z$, etc.) to form rational functions, and then proceeds by applying the algorithm for the rational case. 
\end{remark}

\vspace{0.3 cm}
In order to illustrate the performance of the procedure {\tt AffineEquivalences}, we consider now two examples where we compute the affine equivalences between curves taken from Ex. \ref{somecurves}, and the images of these curves under an affine mapping. These examples were computed with {\tt Maple} and executed in a PC with a 3.60 GHz Intel Core i7 processor and 32 GB RAM, and are accessible in \cite{website} as well.

\begin{example}\label{Ex1cont}
[{\it 2D catenary curves}] Consider the curves $\bfC_1$ and $\bfC_2$ parametrized by 
\begin{equation*}
\bfp(z)=\begin{pmatrix}
2z-\cosh(2z)+1\\
4z+\cosh(2z)
\end{pmatrix},\,\,\bfq(z)=\begin{pmatrix}
z\\
\cosh(z)
\end{pmatrix}.
\end{equation*}
The curve $\bfq(z)$ corresponds to the first curve in Ex. \ref{somecurves}, which is a catenary curve. After appying our algorithm, we find two factors $\widehat{H}_i,(z,\omega)$, $i=1,2$, common to the $M_j$, namely
\[
\widehat{H}_1(z,\omega)=\cosh(\omega)\sinh(2z)-\cosh(2z)\sinh(\omega),\mbox{ }\widehat{H}_2(z,\omega)=\cosh(\omega)\sinh(2z)+\cosh(2z)\sinh(\omega).
\]
When solving for $\omega$, we get infinitely many (complex) M\"obius functions leading to infinitely many (complex) affine equivalences, which reveals that the $\widehat{H}_i(\omega,z)$ contain M\"obius-like factors. The affine equivalences can be classified in three classes $f_j(\bfx)=A_j\bfx+\bfb_j$, $j\in\{1,2,3\}$, with associated M\"obius functions $\varphi_j(z)$:
\begin{equation*}
A_1=\begin{pmatrix}
\dfrac{\strut 1}{\strut 3} & \dfrac{\strut 1}{\strut 3}\\
(-1)^{k_1+1}\dfrac{\strut 2}{\strut 3} & (-1)^{k_1}\dfrac{\strut 1}{\strut 3}
\end{pmatrix},\,\,\bfb_1=\begin{pmatrix}
-\dfrac{\strut 1}{\strut 3}+{\bf i}k_1\pi\\
(-1)^{k_1}\dfrac{\strut 2}{\strut 3}
\end{pmatrix},\,\,\varphi_1(z)=2z+{\bf i}k_1\pi,\,k_1\in\mathbb{Z},
\end{equation*}

\begin{align*}
A_2=\begin{pmatrix}
-\dfrac{\strut 1}{\strut 3} & -\dfrac{\strut 1}{\strut 3}\\
-\dfrac{\strut 2}{\strut 3} & \dfrac{\strut 1}{\strut 3}
\end{pmatrix},\,\,\bfb_2=\begin{pmatrix}
\dfrac{\strut 1}{\strut 3}+2{\bf i}k_2\pi\\
\dfrac{\strut 2}{\strut 3}
\end{pmatrix},\,\,\varphi_2(z)=-2z+{\bf i}k_2\pi,\,k_2\in\mathbb{Z},
\end{align*}
and
\begin{align*}
A_3=\begin{pmatrix}
-\dfrac{\strut 1}{\strut 3} & -\dfrac{\strut 1}{\strut 3}\\
\dfrac{\strut 2}{\strut 3} & -\dfrac{\strut 1}{\strut 3}
\end{pmatrix},\,\,\bfb_3=\begin{pmatrix}
\dfrac{\strut 1}{\strut 3}+(2k_2+1){\bf i}\pi\\
-\dfrac{\strut 2}{\strut 3}
\end{pmatrix},\,\,\varphi_3(z)=-2z+(2k_2+1){\bf i}\pi,\,k_2\in\mathbb{Z},
\end{align*}
where ${\bf i}^2=-1$. If we just consider real affine equivalences, we have three of them, which correspond to fixing $k_1=0$ for $f_1(\bfx)$, $k_2=0$ for $f_2(\bfx)$, $k_2=-1/2$ for $f_3(\bfx)$. The whole computation took 0.172 seconds.
\end{example}

\begin{example}\label{Ex1cont2}
[{\it 3D spirals}] Consider the curves $\bfC_1$ and $\bfC_2$ parametrized by 
\begin{equation*}
\bfp(z)=\begin{pmatrix}
z\dfrac{\strut e^{4{\bf i}z}+1}{\strut e^{2{\bf i}z}}-{\bf i}z\dfrac{\strut e^{4{\bf i}z}-1}{\strut e^{2{\bf i}z}}+1\\
2z\dfrac{\strut e^{4{\bf i}z}+1}{\strut e^{2{\bf i}z}}-{\bf i}z\dfrac{\strut e^{4{\bf i}z}-1}{\strut e^{2{\bf i}z}}-2z\\
-2z-1
\end{pmatrix} ,\,\,\bfq(z)=\begin{pmatrix}
z\dfrac{\strut e^{2{\bf i}z}+1}{\strut 2e^{{\bf i}z}}\\
-{\bf i}z\dfrac{\strut e^{2{\bf i}z}-1}{\strut 2e^{{\bf i}z}}\\
z
\end{pmatrix}.
\end{equation*}
The curve $\bfq(z)$ corresponds to the third curve in Ex. \ref{somecurves}, which is a 3D spiral. After applying our algorithm, we find two M\"obius-like factors $H_i(z,\omega)$, $i=1,2$, common to the $M_j(z,\omega)$, namely
\[
H_1(z,\omega)=\omega-2z,\mbox{ }H_2(z,\omega)=\omega+2z.
\]
When solving for $\omega$, we get two M\"obius transformations $\varphi_1(z)=-2z$ and $\varphi_2(z)=2z$ corresponding to the affine equivalences $f_1(\bfx)=A_1\bfx+\bfb_1$ and $f_2(\bfx)=A_2\bfx+\bfb_2$ with
\begin{equation*}
A_1=\begin{pmatrix}
1 & -1 &1\\
2 & -1 & 1\\
0 & 0 & 1
\end{pmatrix},\,\,\bfb_1=\begin{pmatrix}
0\\
-1\\
1
\end{pmatrix}
\end{equation*}
and
\begin{equation*}
A_2=\begin{pmatrix}
-1 & 1 &-1\\
2 & -1 & 1\\
0 & 0 & -1
\end{pmatrix},\,\,\bfb_1=\begin{pmatrix}
0\\
-1\\
-1
\end{pmatrix}.
\end{equation*}
The whole computation took 0.032 seconds.
\end{example}

\begin{example}
[{\it Rational curves in $n$-th dimension}] Finally, in Table \ref{tablerat} we present the results of performance tests to compute affine equivalences between rational curves of various degrees, in different dimensions. The rational curves in the experiments were randomly generated, see also \cite{website}, with coefficients between $-10$ and $10$. After generating the first curve, the second curve was obtained by applying an affine mapping $f(\bfx)=A\bfx+\bfb$ to the first curve, where the matrix and the translation vector, for each dimension, are shown in Table \ref{tableAff}; additionally, the resulting curve was reparametrized using a M\"obius transformation $\varphi(z)=2z-1$. The timings to recover the affine equivalences are shown in Table \ref{tableAff}: the rows of Table \ref{tableAff} correspond to dimensions from $n=2$ to $n=6$, and the columns, to degrees from $d=6$ to $d=12$. For degrees up to 10, we can compute the affine equivalences between the curves in less than a minute, for all the dimensions tested. 
 
\begin{table}[H]
\centering
\begin{tabular}{c|c c}
$n$ & $A$ & $\bfb$\\
\hline\\
$2$ & $\begin{pmatrix}
1 & -1 \\
2 & 0
\end{pmatrix}$ & $\begin{pmatrix}
0\\
1
\end{pmatrix}$\\
&&\\
$3$ & $\begin{pmatrix}
1 & -1 & 2 \\
2 & 0 & 3\\
0 & 0 & 4
\end{pmatrix}$ & $\begin{pmatrix}
0\\
1\\
0
\end{pmatrix}$\\
&&\\
$4$ & $\begin{pmatrix}
1 & -1 & 2 & -1\\
2 & 0 & 3 & 0\\
0 & 0 & 4 & -1 \\
0 & 1 & 0 & 2
\end{pmatrix}$ & $\begin{pmatrix}
0\\
1\\
0\\
0
\end{pmatrix}$\\
&&\\
$5$ & $\begin{pmatrix}
1 & -1 & 2 & -1 &3 \\
2 & 0 & 3 & 0 & 1\\
0 & 0 & 4 & -1 & 3 \\
0 & 0 & 4 & -1 & 0\\
0 & 1 & 0 & 2 &1
\end{pmatrix}$ & $\begin{pmatrix}
0\\
1\\
0\\
0\\
0
\end{pmatrix}$\\
&&\\
$6$ & $\begin{pmatrix}
1 & -1 & 2 & -1 &3 & 0 \\
2 & 0 & 3 & 0 & 1 & 2\\
0 & 0 & 4 & -1 & 3 & 1\\
0 & 0 & 4 & -1 & 0 & 2\\
0 & 1 & 0 & 2 &1 & 1\\
0 & -1 & 2 & 0 & -1 & 3
\end{pmatrix}$ & $\begin{pmatrix}
0\\
1\\
0\\
0\\
0\\
0
\end{pmatrix}$
\end{tabular}
\caption{Affine mappings used in the examples}
\label{tableAff}
\end{table}

\begin{table}[H]
\centering
\begin{tabular}{c c|c c c c c c c}
&&\multicolumn{7}{c}{Degree}\\
&$n$ & $6$ & $7$ & $8$ & $9$ & $10$ & $11$ & $12$ \\
\hline
 &$2$ &$0.109$&$0.188$&$0.125$&$0.203$&$0.453$&$0.750$&$0.532$\\
&$3$&$0.969$&$1.969$&$3.750$&$6.406$&$8.579$&$12.281$&$15.703$\\
&$4$&$1.343$&$2.063$&$4.359$&$7.453$&$12.531$&$17.688$&$30.234$\\
&$5$&$2.813$&$6.047$&$14.000$&$28.609$&$48.406$&$89.203$&$138.546$\\
&$6$&$0.922$&$6.281$&$12.203$&$26.609$&$51.328$&$90.344$&$153.719$
\end{tabular}
\caption{CPU time in seconds for affine equivalences of random rational curves with various degrees in various dimensions}
\end{table}
\label{tablerat}

\end{example}

\section{Justification of step (ii): computation of M\"obius-commuting invariants} \label{difficult}

Step (ii) in the strategy presented at the beginning of Section \ref{overall} corresponds to the computation of what we called M\"obius-commuting invariants. The deduction of these invariants is involved; so in order to develope our reasoning, we will distinguish three small substeps, that we present in separate subsections:
 
\begin{itemize}
\item [(ii.1)] Rewriting high order derivatives of $\omega$ and rewriting Fa\`a di Bruno's formula. 
\item [(ii.2)] Expansion of $I_i(\bfq(\omega))$ for $i=1,\ldots,n-1$. 
\item [(ii.3)] Computation of M\"obius-commuting invariants.
\end{itemize}

Here we recall the notation $\omega=\varphi(z)$, and the notation $\omega',\omega'',\ldots, \omega^{(k)}$ for the derivatives of $\omega$ with respect to $z$. Recall also from Subsection \ref{secondstep} that the rough idea in step (ii) is to eliminate the derivatives of $\omega$ from the expressions resulting from expanding $I_i(\bfq(\omega))$. In order to do that, first we will rewrite all the derivatives $\omega^{(k)}$ in terms of just $\omega'$; we will do that in step (ii.1) with the help of the expansion of $I_n(\bfq(\omega))$, for which we will make use of the tools introduced in Subsection \ref{addtools} and, in particular, Fa\`a di Bruno's formula. Then, in step (ii.2), we will compute an expansion form for $I_i(\bfq(\omega))$ for $i=1,\ldots,n-1$; this is the hardest part, where we will need to make use, again, of Fa\`a di Bruno's formula, rewritten in an advantageous form in substep (ii.1), and some combinatorics. Finally, in step (ii.3), we will make use of $I_0(\bfq(\omega))$, the auxiliary invariant introduced in Eq. \eqref{I3} to finally eliminate $\omega'$, the only derivative of $\omega$ left after substep (ii.1), and compute the M\"obius-commuting invariants.

\subsection{Substep \emph{(ii.1)}}
%\subsection{Expansion of $I_n(\bfq(\omega))$ and high order derivatives of $\omega$}

Our first step in order to eliminate the derivatives $\omega^{(k)}$ is to write all of them in terms of $\omega'$; later, we will use this to rewrite Fa\`a di Bruno's formula, introduced in Section \ref{addtools}, in an alternative form that will be useful in substep (ii.2). In order to do this, we will take advantage of the expansion of $I_n(\bfq(\omega))$: in general, expanding $I_i(\bfq(\omega))$ for $i=0,1,\ldots,n-1$ is messy and will be the hardest part, deferred for substep (ii.2), but expanding $I_n(\bfq(\omega))$ is much more accesible. So let focus on this. From Eq. \eqref{first-inv}, we need to expand $\Delta(\bfq(\omega))$ and $A_n(\bfq(\omega))$. In both cases we will make use of Fa\`a di Bruno's formula. First, from Eq. \eqref{As} 

\begin{equation}\label{Deltaqomega}
\Delta(\bfq(\omega))=\left\Vert \dfrac{d}{dz}(\bfq(\omega)),\dfrac{d^2}{dz^2}(\bfq(\omega)),\ldots,\dfrac{d^n}{dz^n}(\bfq(\omega))\right\Vert.
\end{equation}
From Fa\`a di Bruno's formula, Eq. \eqref{someq1} and Eq. \eqref{someq2}, we have 
\[
\dfrac{d}{dz}(\bfq(\omega))=\omega'\bfq'(\omega), 
\]
and for $k\geq 2$,
\begin{equation}\label{dzk}
\dfrac{d^k}{dz^k}(\bfq(\omega))=\bullet_{k-1}+(\omega')^k\bfq^{(k)}(\omega),
\end{equation}
where $\bullet_{k-1}$ is a linear combination of the derivatives of $\bfq$ up to order $k-1$, evaluated at $\omega$. By expanding the determinant in Eq. \eqref{Deltaqomega} as a sum of determinants, we observe that all the determinants including terms of the $\bullet_{k-1}$, $k=2,\ldots,n$, must be zero. Thus, we are left with one determinant, whose columns $k=1,2,\ldots,n$ are 
\[
(\omega')^k\bfq^{(k)}(\omega),
\]
and we get the following result. 

\begin{lemma}\label{deltaq}
$\Delta(\bfq(\omega))=(\omega')^{\frac{n(n+1)}{2}}\Delta \bfq(\omega).$
\end{lemma} 

To expand $I_n(\bfq(\omega))$ we also need $A_n(\bfq(\omega))$. From Eq. \eqref{As},
\begin{equation}\label{AnExp1}
A_n(\bfq(\omega))=\left\Vert \dfrac{d(\bfq(\omega))}{dz},\ldots,\dfrac{d^{n-1}(\bfq(\omega))}{dz^{n-1}},\dfrac{d^{n+1}(\bfq(\omega))}{dz^{n+1}}\right\Vert.
\end{equation}
Furthermore, from Fa\`a di Bruno's formula, 
\begin{equation}\label{AnExp4}
\dfrac{d^{n+1}(\bfq(\omega))}{dz^{n+1}}=\bullet_{n-1}+\bfq^{(n)}(\omega)B_{n+1,n}(\omega',\omega'')+\bfq^{(n+1)}(\omega)B_{n+1,n+1}(\omega'),
\end{equation}
where $\bullet_{n-1}$ contains a linear combination of the derivatives of $\bfq$ up to order $n-1$, evaluated at $\omega$. Also, from Eq. \eqref{bell}, 
\begin{equation}\label{AnExp2}
\bfq^{(n)}(\omega)B_{n+1,n}(\omega',\omega'')=\dfrac{n(n+1)}{2}(\omega')^{n-1}\omega''\bfq^{(n)}(\omega),
\end{equation}
and
\begin{equation}\label{AnExp3}
\bfq^{(n+1)}(\omega)B_{n+1,n+1}(\omega')=(\omega')^{n+1}\bfq^{(n+1)}(\omega).
\end{equation}

Now let us substitute Eq. \eqref{dzk}, for $k=1,\ldots,n-1$, and Eqs. \eqref{AnExp4}-\eqref{AnExp2}-\eqref{AnExp3}, into Eq. \eqref{AnExp1}. Expanding as a sum of determinants, we observe again that the determinants including some term of the $\bullet_{k-1}$, $k=2,\ldots,n$ must be zero. Thus, we arrive at
\[
A_n(\bfq(\omega))=\dfrac{n(n+1)}{2}(\omega')^{\frac{n(n+1)}{2}-1}\omega''\Delta(\bfq)(\omega)+(\omega')^{\frac{n(n+1)}{2}+1}A_n(\bfq)(\omega).
\]
Dividing the last equality by $\Delta(\bfq(\omega))$ and using Lemma \ref{deltaq}, we reach the following result. 

\begin{lemma}\label{Inqomega}
$\displaystyle{I_{n}(\bfq(\omega))=\dfrac{n(n+1)}{2}\dfrac{\omega''}{\omega'}+\omega'I_n(\bfq)(\omega)}.$
\end{lemma}

Notice that the formula for $I_n(\bfq(\omega))$ in Lemma \ref{Inqomega} is linear in $\omega''$. This allows us to write
\begin{equation}\label{omegados}
\omega''=\dfrac{2\omega'}{n(n+1)}(I_{n}(\bfq(\omega))-\omega'I_n(\bfq)(\omega)).
\end{equation}
Furthermore, using Eq. \eqref{omegados} and invoking Lemma \ref{schwa} in Section \ref{addtools}, we can write all the derivatives of $\omega$ in terms of just $\omega'$, which was one of the goals of this substep. We formulate this as a corollary of Lemma \ref{Inqomega}. 

\begin{corollary}\label{comCor2}
$\omega^{(k)}=\dfrac{k!}{n^{k-1}(n+1)^{k-1}}\omega'(I_n(\bfq(\omega))-\omega'I_n(\bfq)(\omega))^{k-1}$, $k\geq 2$.
\end{corollary} 

Finally, let us use Corollary \ref{comCor2} to rewrite Fa\`a di Bruno's formula. In order to do this, let us introduce the notation 
\begin{equation}\label{newnota}
K_n:=I_n(\bfq(\omega)),\mbox{ }G_n:=I_n(\bfq)(\omega),\mbox{ }\Phi:=K_n-\omega' G_n.
\end{equation}
Thus, using Corollary \ref{comCor2} and the above notation, the Bell polynomial in Eq. \eqref{bell} can be written as
\begin{equation}\label{newBell}
B_{k,m}(\omega',\omega'',\ldots,\omega^{(k+1-m)})=B_{k,m}\left(\omega',\dfrac{2!}{n(n+1)}\omega' \Phi,\ldots,\ldots,\dfrac{(k+1-m)!}{n^{k-m}(n+1)^{k-m}}\omega'\Phi^{k-m}\right),
\end{equation}
where only $\omega'$ is involved. Let 
\[
L(k,m)=\binom{k-1}{m-1}\dfrac{k!}{m!}
\]
be the \emph{Lah number} $L(k,m)$ (see for instance \cite{K21}), and let us denote
\begin{equation}\label{bellnew}
\tilde{B}_{k,m}=\dfrac{1}{n^{k-m}(n+1)^{k-m}}L(k,m).
\end{equation}
Then we have the following result, proved in Appendix I. 

\begin{lemma}\label{comLem4}
$B_{k,m}(\omega',\omega'',\ldots,\omega^{(k+1-m)})=\tilde{B}_{k,m}\omega'^m \Phi^{k-m}$.
\end{lemma}

Using Lemma \ref{comLem4} we can rewrite Fa\`a di Bruno's formula as 
\begin{equation}\label{newformula}
\dfrac{d^k(\bfq(\omega))}{dt^k}=\sum_{m=1}^{k}{\tilde{B}_{k,m}{\omega'}^{m}\Phi^{k-m}\bfq^{(m)}(\omega)}.
\end{equation}
We will use this in the next substep. Additionally, notice that $\tilde{B}_{k,k}=1$; also, we will assume the convention that $\tilde{B}_{k,m}=0$ when $k<m$.

\subsection{Substep \emph{(ii.2)}} 

The goal of this substep is to expand the $I_i(\bfq(\omega))$, for $i=1,\ldots,n-1$, defined in Eq. \eqref{first-inv}. Since from Lemma \ref{deltaq} we already know $\Delta(\bfq(\omega))$, we need to analyze, for $i\neq n$, 
\begin{equation}\label{As-again}
A_i(\bfq(\omega))=\left\Vert \underbrace{\dfrac{d(\bfq(\omega))}{dz},\ldots,\dfrac{d^{i-1}(\bfq(\omega))}{dz^{i-1}}}_{i-1},\dfrac{d^{n+1}(\bfq(\omega))}{dz^{n+1}},\underbrace{\dfrac{d^{i+1}(\bfq(\omega))}{dz^{i+1}},\ldots,\dfrac{d^{n}(\bfq(\omega))}{dz^{n}}}_{n-i}\right\Vert.
\end{equation}

By using Fa\`a di Bruno's formula in Eq. \eqref{newformula}, we know that each column, i.e. each derivative, in Eq. \eqref{As-again} is a linear combination of derivatives, where the number of terms is equal to the order of the derivative. Furthermore, we observe that for the first $i-1$ columns, only the last term, i.e. the term involving the highest derivative, matters, since the other terms lead to vanishing determinants when expanding $A_i(\bfq(\omega))$ as a sum of determinants. Thus, we can write

\begin{equation}\label{As-again2}
A_i(\bfq(\omega))=\left\Vert \underbrace{\omega'\bfq'(\omega),\ldots,(\omega')^{i-1}\bfq^{(i-1)}(\omega)}_{i-1},\dfrac{d^{n+1}(\bfq(\omega))}{dz^{n+1}},\underbrace{\dfrac{d^{i+1}(\bfq(\omega))}{dz^{i+1}},\ldots,\dfrac{d^{n}(\bfq(\omega))}{dz^{n}}}_{n-i}\right\Vert.
\end{equation}

Next let us express Eq. \eqref{As-again2} as a sum of determinants, using Fa\`a di Bruno's formula in Eq. \eqref{newformula}. In order to do this, we consider the following two subsets: 
\begin{itemize}
\item ${\mathcal P}$ represents the set consisting of the permutations of $\{i,i+1,\ldots,n\}$. Notice that $\#\{i,i+1,\ldots,n\}=n-i+1$, where $\#$ denotes here the cardinal of a set. 
\item ${\mathcal Q}$ represents the set consisting of the combinations of $n-i$ items from $\{i,i+1,\ldots,n\}$; thus, each element of ${\mathcal Q}$ skips exactly one element, $r\in\{i,i+1,\ldots,n\}$. We will denote, also, the set of the permutations of the set $\{i,i+1,\ldots,n\}-\{r\}$ by ${\mathcal P}_r$.  
\end{itemize}

Now when expanding $A_i(\bfq(\omega))$ as a sum of determinants, and ignoring vanishing determinants, we get sums of terms like
\begin{equation}\label{distinctterms}
\left \Vert \omega'\bfq'(\omega),\ldots,(\omega')^{i-1}\bfq^{(i-1)}(\omega), \Box_{\ell_1}\bfq^{(\ell_1)}(\omega),\Box_{\ell_2}\bfq^{(\ell_2)}(\omega),\ldots,\Box_{\ell_{n-i+1}}\bfq^{(\ell_{n-i+1})}(\omega)\right\Vert,
\end{equation}
where $1,\ldots,i-1,\ell_1,\ell_2,\ldots,\ell_{n-i+1}$ are $n$ different numbers (otherwise the determinant is zero), the $\Box_{\ell_k}$ are coefficients to be discussed, and $\{1,\ldots,i-1,\ell_1,\ell_2,\ldots,\ell_{n-i+1}\}$ is a subset of $\{1,2,\ldots,i-1,i,i+1,\ldots,n,n+1\}$. We are going to split the determinants in Eq. \eqref{distinctterms} into two different groups, according to $\ell_1\neq n+1$ or $\ell_1=n+1$. The sums of the determinants of each group are denoted $A^{(I)}_i(\bfq(\omega))$, $A^{(II)}_i(\bfq(\omega))$, so $A_i(\bfq(\omega))=A^{(I)}_i(\bfq(\omega))+A^{(II)}_i(\bfq(\omega))$.

\vspace{0.3 cm}
\noindent\fbox{$\ell_1\neq n+1$}

\vspace{0.3 cm}
Here we have the determinants where we do not pick the term in $\bfq^{(n+1)}(\omega)$ of the expansion, using Fa\`a di Bruno's formula, of $\dfrac{d^{n+1}(\bfq(\omega))}{dz^{n+1}}$, the $i$-th colum of Eq. \eqref{As-again2}. In this case, $\{\ell_1,\ell_2,\ldots,\ell_{n-i+1}\}\in {\mathcal P}$; thus, the columns of the determinant in Eq. \eqref{distinctterms} coincide, up to a permutation, with the columns of $\Delta \bfq(\omega)$, so Eq. \eqref{distinctterms} is equal to 
\begin{equation}\label{det1}
\alpha \Vert \bfq'(\omega),\ldots,\bfq^{(i-1)}(\omega),\bfq^{(\ell_1)}(\omega),\ldots,\bfq^{(\ell_{n-i+1})}(\omega)\Vert =\alpha (-1)^\sigma \Delta \bfq(\omega),
\end{equation}
where $\alpha$ is a coefficient to be discussed, and $(-1)^\sigma$ denotes the signature of the permutation 
\[
\{i,i+1,\ldots,n\} \stackrel{\sigma}{\to} \{\ell_1,\ell_2,\ldots,\ell_{n-i+1}\}.
\]
We have exactly $\#{\mathcal P}$ terms like this in the expansion of $A_i(\bfq(\omega))$. 

Let us now discuss the value of the coefficient $\alpha$. In Eq. \eqref{distinctterms}, for $k=1,\ldots,i-1$ each derivative $\bfq^{(k)}(\omega)$ appears multiplied by $(\omega')^k$. Additionally, from Eq. \eqref{newformula}, we deduce that the derivative $\bfq^{(\ell_1)}(\omega)$ appears multiplied by $\tilde{B}_{n+1,\ell_1}\Phi^{(n+1)-\ell_1}$, the derivative $\bfq^{(\ell_2)}(\omega)$ appears multiplied by $\tilde{B}_{i+1,\ell_2}\Phi^{(i+1)-\ell_2}$, etc. and the derivative $\bfq^{\ell_{n-i+1}}(\omega)$ appears multiplied by $\tilde{B}_{n,\ell_{n-i+1}}\Phi^{n-\ell_{n-i+1}}$. Taking these constants out of the determinant in Eq. \eqref{distinctterms}, we get that 
\[
\alpha= (\omega')^{\frac{n(n+1)}{2}} \tilde{B}_{n+1,\ell_1}\tilde{B}_{i+1,\ell_2}\cdots \tilde{B}_{n,\ell_{n-i+1}} \Phi^{(n+1)-i}.
\]
Here we need to take into account that: 
\begin{itemize}
\item [(1)] The power of $\omega'$ in the above expression comes from 
\[
\omega'(\omega')^2\cdots (\omega')^{i-1}(\omega')^{\ell_1}\cdots (\omega)^{\ell_{n-i+1}}.
\]
Since $\{\ell_1,\ell_2,\ldots,\ell_{n-i+1}\}\in {\mathcal P}$, the sum of $1,2,\ldots,i-1,\ell_1,\ldots,\ell_{n-i+1}$ coincides with the sum of the first $n$ natural numbers. 
\item [(2)] The power of $\Phi$ is the result of the sum 
\begin{equation}\label{sumN}
[(n+1)-\ell_1]+[(i+1)-\ell_2]+\cdots +[n-\ell_{n-i+1}]=n+1-i.
\end{equation}
\end{itemize}
Thus, when summing over the elements of ${\mathcal P}$, the sum of the determinants in the expansion of $A_i(\bfq(\omega))$ with $\ell_1\neq n+1$ yields 
\begin{equation}\label{firsterm}
A^{(I)}_i(\bfq(\omega))=\sum_{\sigma \in {\mathcal P}}(-1)^{\sigma} (\omega')^{\frac{n(n+1)}{2}} \tilde{B}_{n+1,\ell_1}\tilde{B}_{i+1,\ell_2}\cdots \tilde{B}_{n,\ell_{n-i+1}} \Phi^{(n+1)-i} \Delta \bfq(\omega),
\end{equation}
which we can also write as
\begin{equation}\label{secondterm}
A^{(I)}_i(\bfq(\omega))=(\omega')^{\frac{n(n+1)}{2}} \Phi^{(n+1)-i} \Delta \bfq(\omega) \underbrace{\sum_{\sigma \in {\mathcal P}}(-1)^{\sigma} \tilde{B}_{n+1,\ell_1}\tilde{B}_{i+1,\ell_2}\cdots \tilde{B}_{n,\ell_{n-i+1}}}.
\end{equation}

The underbraced expression corresponds exactly to the definition of an $(n-i+1)\times (n-i+1)$ determinant, whose $j$-th column consists of the values of $\tilde{B}_{n+1,n+1-j},\tilde{B}_{n,n+1-j},\ldots,\tilde{B}_{i+1,n+1-j}$, where $\tilde{B}_{k,m}=0$ when $k<m$, and $\tilde{B}_{k,k}=1$. We call this determinant $M^{n+1,i+1}$, so   
\begin{equation}\label{thirdterm}
A^{(I)}_i(\bfq(\omega))=(\omega')^{\frac{n(n+1)}{2}} \Phi^{(n+1)-i} \Delta \bfq(\omega) M^{n+1,i+1}.
\end{equation}

\vspace{0.3 cm}
\noindent\fbox{$\ell_1= n+1$}

\vspace{0.3 cm}
Here we have the determinants where we pick the term in $\bfq^{(n+1)}(\omega)$ of the expansion, using Fa\`a di Bruno's formula, of $\dfrac{d^{n+1}(\bfq(\omega))}{dz^{n+1}}$, the $i$-th colum of Eq. \eqref{As-again2}. In other words, we have determinants like Eq. \eqref{distinctterms}, where the $i$-th column is $(\omega')^{n+1}\bfq^{(n+1)}(\omega)$, and $\{\ell_2,\ldots, \ell_{n-i+1}\}\in {\mathcal Q}$, so Eq. \eqref{distinctterms} is equal to 
\begin{equation}\label{det2}
\beta \Vert \bfq'(\omega),\ldots,\bfq^{(i-1)}(\omega),\bfq^{(n+1)}(\omega),\bfq^{(\ell_2)},\ldots,\bfq^{(\ell_{n-i+1})}(\omega)\Vert,
\end{equation}
where $\alpha$ is a coefficient to be discussed. Since each combination of ${\mathcal Q}$ skips exactly one element of $\{i,i+1,\ldots,n\}$, say $r$, the value that we have in the above determinant, leaving $\alpha$ aside, coincides with $A_r(\bfq)(\omega)$ up to $(-1)^{r}(-1)^\sigma$. The factor $(-1)^{r}$ comes from the fact that to get $A_r(\bfq)(\omega)$ we need to take $\bfq^{(n+1)}(\omega)$ to the $\ell_r$-th column, while the factor $(-1)^\sigma$ denotes the signature of a permutation of the elements in $\{i,i+1,\ldots,n\}-\{r\}$ so that we reach exactly $A_r(\bfq)(\omega)$. Thus, Eq. \eqref{det2} can be written as 
\begin{equation}\label{det3}
\beta (-1)^{r}(-1)^\sigma A_r(\bfq)(\omega).
\end{equation}

To discuss the value of the coefficient $\beta$, we argue as in the case $\ell_1\neq n+1$. The difference is that in Eq. \eqref{distinctterms}, using Eq. \eqref{newformula} we have $\ell_1=n+1$ and the derivative $\bfq^{(n+1)}(\omega)$ appears multiplied by $(\omega')^{n+1}\tilde{B}_{n+1,n+1}$. Again arguing as in the case $\ell_1\neq n+1$, we get that
\[
\beta= (\omega')^L \tilde{B}_{n+1,n+1}\tilde{B}_{i+1,\ell_2}\cdots \tilde{B}_{n,\ell_{n-i+1}} \Phi^N.
\]
The power $L$ equals the sum of $1,2,\ldots,n+1$ minus $r$, so
\[
L=\dfrac{n(n+1)}{2}+n+1-r.
\]
On the other hand, 
\[
N=[(i+1)-\ell_2]+\cdots+[n-\ell_{n-i+1}]=[(i+1)+\cdots +n]-[\ell_2+\cdots+\ell_{n-i+1}].
\]
Since $\{\ell_2,\ldots, \ell_{n-i+1}\}\in {\mathcal Q}$, and the list $\{\ell_2,\ldots, \ell_{n-i+1}\}$ is the result of removing $r$ from the list $\{i,i+1,\ldots,n\}$ and reordering it, $\ell_2+\cdots+\ell_{n-i+1}=[i+(i+1)+\cdots +n]-r$, and hence $N=r-i$. Thus, Eq. \eqref{det3} turns into
\begin{equation}\label{det4}
(\omega')^{\frac{n(n+1)}{2}} (-1)^{r} (\omega')^{n+1-r} \Phi^{r-i} A_r(\bfq)(\omega) (-1)^\sigma \tilde{B}_{n+1,n+1}\tilde{B}_{i+1,\ell_2}\cdots \tilde{B}_{n,\ell_{n-i+1}}.
\end{equation}

Now to get the sum of all the determinants corresponding to $\ell_1=n+1$, i.e. $A^{(II)}_i(\bfq(\omega))$, we need to sum Eq. \eqref{det4} over the permutations ${\mathcal P}_r$, and then over the combinations ${\mathcal Q}$. However, this last sum is nothing else than the sum over $r$, from $r=i$ to $r=n$. Thus, we get that 
\begin{equation}\label{det5}
A^{(II)}_i(\bfq(\omega))= (\omega')^{\frac{n(n+1)}{2}} \sum_{r=i}^n  (-1)^{r} (\omega')^{n+1-r} \Phi^{r-i} A_r(\bfq)(\omega) \underbrace{\sum_{\sigma \in{\mathcal P}_r} (-1)^\sigma \tilde{B}_{n+1,n+1}\tilde{B}_{i+1,\ell_2}\cdots \tilde{B}_{n,\ell_{n-i+1}}}.
\end{equation}

The underbraced expression corresponds, as in the case $\ell_1\neq n+1$, to an $(n-i+1)\times (n-i+1)$ determinant $M^{n+1,i+1}_r$ defined in the following way:
\begin{itemize}
\item If $j<r$, the $j$-th column of $M^{n+1,i+1}_r$ consists of the values $\tilde{B}_{n+1,n+1-j},\tilde{B}_{n,n+1-j},\ldots,\tilde{B}_{i+1,n+1-j}$. 
\item If $j\geq r$, the $j$-th column of $M^{n+1,i+1}_r$ consists of the values $\tilde{B}_{n+1,n-j},\tilde{B}_{n,n-j},\ldots,\tilde{B}_{i+1,n-j}$, 
\end{itemize}
where we recall that $\tilde{B}_{k,k}=1$, and $\tilde{B}_{k,m}=0$ when $k<m$. We will assume the convention 
\begin{equation}\label{M1}
M^{n+1,i+1}_0=M^{n+1,i+1},\mbox{ and }M^{n+1,i+1}_{n+1}=(-1)^n,
\end{equation}
where $M^{n+1,i+1}$, without any subindex, was the underbraced determinant in Eq. \eqref{secondterm}. Furthermore, one can check that 
\begin{equation}\label{M2}
M_{n-i+1}^{n+1,i+1}=M^{n,i+1}.
\end{equation}

Therefore, 
\begin{equation}\label{fourthterm}
A^{(II)}_i(\bfq(\omega))=(\omega')^{\frac{n(n+1)}{2}} \sum_{r=i}^n  (-1)^{r} (\omega')^{n+1-r} \Phi^{r-i} A_r(\bfq)(\omega)M^{n+1,i+1}_r.
\end{equation}

Since $A_i(\bfq(\omega))=A^{(I)}_i(\bfq(\omega))+A^{(II)}_i(\bfq(\omega))$, finally we get

\begin{equation}\label{fifthterm}
A_i(\bfq(\omega))=(\omega')^{\frac{n(n+1)}{2}} \Phi^{(n+1)-i} \Delta \bfq(\omega) M^{n+1,i+1}+(\omega')^{\frac{n(n+1)}{2}} \sum_{r=i}^n  (-1)^{r} (\omega')^{n+1-r} \Phi^{r-i} A_r(\bfq)(\omega) M^{n+1,i+1}_r,
\end{equation}
and to find $I_i(\bfq(\omega))$, which was the goal of this substep, we just need to divide Eq. \eqref{fifthterm} by $\Delta(\bfq(\omega))$. Taking Lemma \ref{deltaq} into account, we deduce the following result. 

\begin{lemma}\label{developIi}
\begin{equation}\label{sixthterm}
I_i(\bfq(\omega))=\Phi^{(n+1)-i}M^{n+1,i+1}+\sum_{r=i}^n  (-1)^{r} (\omega')^{n+1-r} \Phi^{r-i} I_r(\bfq)(\omega) M^{n+1,i+1}_r.
\end{equation}
\end{lemma}
\noindent
Here we recall, from Eq. \eqref{newnota}, that $\Phi:=K_n-\omega' G_n$, where $K_n:=I_n(\bfq(\omega))$, $G_n:=I_n(\bfq)(\omega)$.

\subsection{Substep \emph{(ii.3)}} 

In this last substep, we will make use of Eq. \eqref{sixthterm}, for $i=1,\ldots,n-1$, to get rid of $\omega'$ and find M\"obius-commuting invariants. Notice that because of the notation introduced in Eq. \eqref{newnota}, when expanding the powers of $\Phi$ we get powers of $\omega'$ as well. We want to reach an expression where the coefficient of each power of $\omega'$ is explicit. To do that, we first express Eq. \eqref{sixthterm} as one sum, using the convention in Eq. \eqref{M1}; here we recover the notation in Eq. \eqref{newnota}, and assume by convention $G_{n+1}:=-1$. Thus,

\begin{equation}\label{seventhterm}
K_i=\sum_{r=i}^{n+1}  (-1)^{r} (\omega')^{n+1-r} (K_n-\omega'G_n)^{r-i} G_r M^{n+1,i+1}_r.
\end{equation}
Calling $p=n-r+1$, we get 
\begin{equation}\label{8thterm}
K_i=\sum_{p=0}^{n-i+1}  (-1)^{n+1-p} (\omega')^{p} (K_n-\omega'G_n)^{n+1-i-p} G_{n+1-p} M^{n+1,i+1}_{n+1-p}.
\end{equation}
Reordering the sum above and using the binomial expansion for $(K_n-s'G_n)^{j}$, we reach
\begin{equation}\label{Ks}
K_i=\sum_{j=0}^{n-i+1}{\sum_{k=0}^{j}{(-1)^k\binom{j}{k}M_j^{n+1,i+1}{\omega'}^{n-i+1-k}G_n^{j-k}K_n^{k}G_{j+i}}},
\end{equation}
where we recall that $G_{n+1}=-1$. 

\vspace{0.3 cm}
Next we need to reorder the indexes above so that the coefficient of the powers of $\omega'$ is explicit. We can do it by rearranging the indices $j$ and $k$ as follows
\begin{equation}\label{Kslast}
K_i=\sum_{k=0}^{n-i+1}{(-1)^k K_n^k \omega'^{n-i+1-k}\sum_{j=k}^{n-i+1}{\binom{j}{k}M_j^{n+1,i+1}G_n^{j-k}G_{j+i}}}.
\end{equation}

We observe that for each $i$, where recall that $i=1,2,\ldots,n-1$, the degree of $K_i$ as a polynomial in $\omega'$ is $n-i+1$. In particular, the degree of $K_1$ is $n$, and the degree of $K_{n-1}$ is $2$. Considering 
\[
\alpha_1:=\omega',\alpha_2:=(\omega')^2,\ldots,\alpha_n:=(\omega')^n,
\]
we notice that when collecting Eq. \eqref{Kslast} for $i=1,\ldots,n-1$, we get a linear system ${\mathcal S}$ in $\alpha_1,\alpha_2,\ldots, \alpha_n$ consisting of $n$ unknowns and $n-1$ equations. However, one can prove that for all $i$ the coefficient of $\alpha_1:=\omega'$ is always zero, i.e., $K_i$ does not depend on $\alpha_1:=\omega'$: indeed, from the formula in Eq. \eqref{Kslast}, for any $i$, the coefficient of $\omega'$ corresponds to the indices $k=n-i$ and $j=n-i,j=n-i+1$, which yields
\[
(-1)^{n-i}G_nK_n^{n-i}\left( M_{n-i}^{n+1,i+1}-(n-i+1)M_{n-i+1}^{n+1,i+1}\right).
\]
But this number is zero, as stated in the following lemma, which is proven in Appendix I.

\begin{lemma}\label{comLem5}
$M_{n-i}^{n+1,i+1}-(n-i+1)M_{n-i+1}^{n+1,i+1}=0$.
\end{lemma}

\vspace{0.3 cm}
Now we are almost ready to find our M\"obius-commuting invariants. Because of Lemma \ref{comLem5}, the system ${\mathcal S}$ is in fact a system with $n$ unknowns, namely $\alpha_2,\ldots,\alpha_n$. Furthermore, ${\mathcal S}$ is upper triangular. Using back substitution, we find the solution of ${\mathcal S}$ as
\begin{equation}\label{solution}
\alpha_k:=(\omega')^{k}=\dfrac{\sum_{i=0}^{k}{M_i^{n+1,n+2-k}K_{n-k+1+i}K_n^{i}}}{\sum_{i=0}^{k}{M_i^{n+1,n-k+2}G_{n-k+1+i}G_n^{i}}},\quad 2\leq k\leq n, 
\end{equation}
where we assume that $K_{n+1}=G_{n+1}=-1$. Note that, for all $k\in\{2,3,\ldots,n\}$, $\alpha_k$ is a rational function of the $K_j,G_j$ whose numerator only depends on the $K_j$ and the denominator only depends on the $G_j$.

\vspace{0.3 cm}
Next we need to invoke the additional invariant $I_0$, introduced in Eq. \eqref{I3}, that we did not use so far. The expansion of $I_0(\bfq(\omega))$ can be directly found as in the case of $I_n(\bfq(\omega))$; denoting $K_0:=I_0(\bfq(\omega))$, $G_0:=I_0(\bfq)(\omega)$, we get that

\begin{equation}\label{K0}
K_0=\dfrac{1}{2}\dfrac{n+2}{n}K_n^2+\omega'^2(G_0-\dfrac{1}{2}\dfrac{n+2}{n}).
\end{equation}

\noindent
Now if we isolate $(\omega')^2$ in Eq. \eqref{K0}, introduce $k=2$ in Eq. \eqref{solution} and set the obtained expressions to be equal, we get that 

\begin{equation}\label{firstfunc}
\dfrac{K_0-\dfrac{1}{2}\dfrac{n+2}{n}K_n^2}{K_{n-1}+M_2^{n+1,n}K_n^2}=\dfrac{G_0-\dfrac{1}{2}\dfrac{n+2}{n}G_n^2}{G_{n-1}+M_2^{n+1,n}G_n^2}.
\end{equation}

\noindent
Getting back to the notation $K_i:=I_i(\bfq(\omega))$ and $G_i:=I_i(\bfq)(\omega)$, we can rewrite Eq. \eqref{firstfunc} as
\begin{equation}\label{firstinv}
\dfrac{I_0(\bfq(\omega))-\dfrac{1}{2}\dfrac{n+2}{n}I_n^2(\bfq(\omega))}{I_{n-1}(\bfq(\omega))+M_2^{n+1,n}I_n^2(\bfq(\omega))}=\dfrac{I_0(\bfq)(\omega)-\dfrac{1}{2}\dfrac{n+2}{n}I_n^2(\bfq)(\omega)}{I_{n-1}(\bfq)(\omega)+M_2^{n+1,n}I_n^2(\bfq)(\omega)}.
\end{equation}

But Eq. \eqref{firstinv} is expressing, exactly, that 
\begin{equation}\label{function}
F_1(\bfq):=\dfrac{I_0(\bfq)-\dfrac{1}{2}\dfrac{n+2}{n}I_n^2(\bfq)}{I_{n-1}(\bfq)+M_2^{n+1,n}I_n^2(\bfq)}
\end{equation} 
is M\"obius-commuting, i.e. that $F_1(\bfq(\omega))=F_1(\bfq)(\omega)$. We can find other invariants in a similar way, by considering $k=3,\ldots,n$ in Eq. \eqref{solution} and eliminating $\omega'$ with the help of Eq. \eqref{K0}. So we have proven the following theorem. 

\begin{theorem}\label{egregium}
Let $F_1(\bfq)$ be the expression in Eq. \eqref{function}, and for $k=3,\ldots,n$, let 
\begin{equation}\label{symfunc}
F_{k-1}(\bfq):=\dfrac{\left(\sum_{i=0}^{k}{M_i^{n+1,n-k+2}I_{n-k+1+i}(\bfq)I_n^{i}(\bfq)}\right)^{e_k/k}}{\left( I_{n-1}(\bfq)+M_2^{n+1,n}I_n^2(\bfq)\right)^{e_k/2}},\quad 3\leq k\leq n,
\end{equation}
where $e_k$ is, for $3\leq k\leq n$, the least common multiple of $2,k$, i.e. $e_k=\emph{lcm}(2,k)$. Then, the $F_\ell$, for $\ell=1,\ldots,n-1$, are M\"obius-commuting, i.e. $F_\ell(\bfq(\omega))=F_\ell(\bfq)(\omega)$. 
\end{theorem}

\section{Conclusion} \label{conclusion}

We have presented an algorithm, generalizing the algorithm in \cite{Gozutok}, to compute the affine equivalences, if any, between two parametric curves in any dimension. Our strategy relies on bivariate factoring, and avoids polynomial system solving. The algorithm works for rational curves and also non-algebraic parametric curves with meromorphic components, admitting a rational inverse. We have implemented the algorithm in {\tt Maple}, and evidence of its performance has been presented. 

The algorithm works whenever not all the M\"obius-commuting invariants are constant. This happens generically, but identifying the curves where this does not occur, as well as providing a solution to the problem for this special case, are questions that we pose here as open problems.

Additionally, in the case of non-algebraic curves, right now we need some hypotheses that are not always satisfied: for instance, planar curves like the cycloid, or the tractrix, or classical planar spirals, do not satisfy our hypotheses. However, we have observed that the algorithm seems to work also for many of those curves, which makes us think that our hypotheses could be relaxed. This requires more theoretical work regarding analytic curves. 

It would be desirable to extend our ideas to the case of rational surfaces$/$hypersurfaces. This probably requires some extra hypotheses, e.g. non-existence of base points or special types of surfaces$/$hypersurfaces, that allow us to guess the type of transformation that we have in the parameter space: such transformation would play a role similar to the role played by M\"obius transformations here. These are questions that we would like to address in the future. 

\section*{Acknowledgements}
Juan Gerardo Alc\'azar is supported by the grant PID2020-113192GB-I00 (Mathematical Visualization: Foundations, Algorithms and Applications) from the Spanish MICINN. Juan G. Alc\'azar is also a member of the Research Group {\sc asynacs} (Ref. {\sc ccee2011/r34}). Uğur Gözütok is supported by the grant $121C421$, in the scope of $2218$-National Postdoctoral Research Fellowship Program, from TUBITAK (The Scientific and Technological Research Council of Türkiye).

\section{Appendix I}

In this appendix we provide the proofs of three results in Section \ref{firstsep} and Section \ref{difficult}. We begin with the proof of Lemma \ref{lemmaI0} in Section \ref{firstsep}.

\begin{proof} (of Lemma \ref{lemmaI0})
Differentiating $I_n$ we get
\begin{equation}\label{diffI3}
\begin{split}
\dfrac{dI_n(u)}{dz}&=\dfrac{\left(\Vert \bm{u}',\cdots,\bfu^{(n-2)},\bm{u}^{(n)},\bfu^{(n+1)}\Vert+\Vert \bm{u}',\cdots,\bm{u}^{(n-1)},\bfu^{(n+2)}\Vert\right)\Vert \bm{u}',\cdots,\bm{u}^{(n)}\Vert}{\Vert \bm{u}',\cdots,\bfu^{(n)}\Vert^2}\\
&-\dfrac{\Vert \bm{u}',\cdots,\bfu^{(n-1)},\bfu^{(n+1)}\Vert\Vert \bm{u}',\cdots,\bfu^{(n-1)},\bfu^{(n+1)}\Vert}{\Vert \bm{u}',\cdots,\bfu^{(n)}\Vert^2}\\
&=-\dfrac{\Vert \bm{u}',\cdots,\bfu^{(n-2)},\bm{u}^{(n+1)}\bfu^{(n)}\Vert}{\Vert \bm{u}',\cdots,\bfu^{(n)}\Vert}\\
&+\dfrac{\Vert \bm{u}',\cdots,\bm{u}^{(n-1)},\bfu^{(n+2)}\Vert}{\Vert \bm{u}',\cdots,\bfu^{(n)}\Vert}\\
&-\dfrac{\Vert \bm{u}',\cdots,\bfu^{(n-1)},\bfu^{(n+1)}\Vert^2}{\Vert \bm{u}',\cdots,\bfu^{(n)}\Vert^2}\\
&=-I_{n-1}+I_{0}-I_n^2.
\end{split}
\end{equation}
Isolating $I_{0}$ from the above equality, we get $I_{0}=\dfrac{dI_n}{dz}+I_{n-1}+I_n^2$. 
\end{proof}

Now we address the proofs of two results in Section \ref{difficult}. First, Lemma \ref{comLem4}.

\begin{proof} (of Lemma \ref{comLem4})
Using Eq. \eqref{bell2} and Eq. \eqref{newBell}, we get that
\begin{equation}\label{init}
B_{k,m}(\omega',\omega'',\ldots,\omega^{(k+1-m)})=\sum{\dfrac{k!}{\ell_1!\cdots \ell_{k+1-m}!}\omega'^{\ell_1}\prod_{i=2}^{k+1-m}{\left(\dfrac{\omega'\Phi^{i-1}\cdot i!}{n^{i-1}(n+1)^{i-1}\cdot i!}\right)^{\ell_i}}},
\end{equation}
where the sum (see Subsection \ref{addtools}) is taken over all sequences $\ell_1,\ell_2,\ldots,\ell_{k+1-m}$ of non-negative integers such that 
\[\sum_{i=1}^{k+1-m}\ell_i=m,\mbox{ }\sum_{i=1}^{k+1-m}i\ell_i=k.\]
In particular,  
\begin{equation}\label{useit}
\sum_{i=1}^{k+1-m}(i-1)\ell_i=\sum_{i=1}^{k+1-m}(i\ell_i-\ell_i)=k-m.
\end{equation}
Using Eq. \eqref{useit}, and since the $i!$ cancels in the numerator and denominator of the product, we have 
\[
\prod_{i=2}^{k+1-m}{\left(\dfrac{\omega'\Phi^{i-1}\cdot i!}{n^{i-1}(n+1)^{i-1}\cdot i!}\right)^{\ell_i}}=\omega'^{\ell_2+\cdots+\ell_{k+1-m}}\dfrac{\Phi^{k-m}}{n^{k-m}(n+1)^{k-m}}.
\]
Substituting this into Eq. \eqref{init}, and since $\ell_1+\cdots+\ell_{k+1-m}=m$, 
\[
B_{k,m}(\omega',\omega'',\ldots,\omega^{(k+1-m)})=\dfrac{1}{n^{k-m}(n+1)^{k-m}}{\omega'}^{m}\Phi^{k-m}\sum{\dfrac{k!}{\ell_1!\cdots \ell_{k+1-m}!}}.
\]
Finally,
\[
\sum{\dfrac{k!}{\ell_1!\cdots \ell_{k+1-m}!}}=\binom{k-1}{m-1}\dfrac{k!}{m!}=L_{m,k},
\]
and the result follows. 
\end{proof}

Finally, we address the proof of Lemma \ref{comLem5}.
\begin{proof} (of Lemma \ref{comLem5})
First, using the properties of factorials we have the following alternative version of the formula in Eq. \eqref{bellnew},
\begin{equation}\label{bellnewnew}
\tilde{B}_{k,m}=\prod_{j=1}^{k-m}{\dfrac{(j+m-1)(j+m)}{jn(n+1)}},
\end{equation} 
where for $k<m$, $\tilde{B}_{k,m}=0$ and for $k=m$, $\tilde{B}_{k,m}=1$.

Next we want to find $M_{n-i}^{n+1,i+1}-(n-i+1)M_{n-i+1}^{n+1,i+1}$. By definition,
\begin{equation}\label{AEq1}
M_{n-i}^{n+1,i+1}=\begin{vmatrix}
\tilde{B}_{n+1,n+1} & \tilde{B}_{n+1,n-1} & \cdots & \tilde{B}_{n+1,i+1} & \tilde{B}_{n+1,i} \\
\tilde{B}_{n,n+1} & \tilde{B}_{n,n-1} & \cdots & \tilde{B}_{n,i+1} & \tilde{B}_{n,i} \\
\tilde{B}_{n-1,n+1} & \tilde{B}_{n-1,n-1} & \cdots & \tilde{B}_{n-1,i+1} & \tilde{B}_{n-1,i} \\
\vdots & \vdots & \ddots & \vdots & \vdots \\
\tilde{B}_{i+2,n+1} & \tilde{B}_{i+2,n-1} & \cdots & \tilde{B}_{i+2,i+1} & \tilde{B}_{i+2,i} \\
\tilde{B}_{i+1,n+1} & \tilde{B}_{i+1,n-1} & \cdots & \tilde{B}_{i+1,i+1} & \tilde{B}_{i+1,i}
\end{vmatrix}.
\end{equation}
Using the definition of $\tilde{B}_{m,k}$, the above determinant is equal to
\begin{equation}\label{AEq2}
M_{n-i}^{n+1,i+1}=\begin{vmatrix}
1 & \prod_{j=1}^{2}{\frac{(j+n-2)(j+n-1)}{jn(n+1)}} & \cdots &\prod_{j=1}^{n-i}{\frac{(j+i)(j+i+1)}{jn(n+1)}} & \prod_{j=1}^{n-i+1}{\frac{(j+i-1)(j+i)}{jn(n+1)}} \\
0 & \prod_{j=1}^{1}{\frac{(j+n-2)(j+n-1)}{jn(n+1)}} & \cdots &\prod_{j=1}^{n-i-1}{\frac{(j+i)(j+i+1)}{jn(n+1)}} & \prod_{j=1}^{n-i}{\frac{(j+i-1)(j+i)}{jn(n+1)}} \\
0 & 1 & \cdots &\prod_{j=1}^{n-i-2}{\frac{(j+i)(j+i+1)}{jn(n+1)}} & \prod_{j=1}^{n-i-1}{\frac{(j+i-1)(j+i)}{jn(n+1)}} \\
\vdots & \vdots & \ddots & \vdots & \vdots \\
0 & 0 & \cdots &\prod_{j=1}^{1}{\frac{(j+i)(j+i+1)}{jn(n+1)}} & \prod_{j=1}^{2}{\frac{(j+i-1)(j+i)}{jn(n+1)}} \\
0 & 0 & \cdots &1 & \prod_{j=1}^{1}{\frac{(j+i-1)(j+i)}{jn(n+1)}}
\end{vmatrix}.
\end{equation}
Analogously,
\begin{equation}\label{AEq3}
M_{n-i+1}^{n+1,i+1}=\begin{vmatrix}
\tilde{B}_{n+1,n} & \tilde{B}_{n+1,n-1} & \cdots & \tilde{B}_{n+1,i+1} & \tilde{B}_{n+1,i} \\
\tilde{B}_{n,n} & \tilde{B}_{n,n-1} & \cdots & \tilde{B}_{n,i+1} & \tilde{B}_{n,i} \\
\tilde{B}_{n-1,n} & \tilde{B}_{n-1,n-1} & \cdots & \tilde{B}_{n-1,i+1} & \tilde{B}_{n-1,i} \\
\vdots & \vdots & \ddots & \vdots & \vdots \\
\tilde{B}_{i+2,n} & \tilde{B}_{i+2,n-1} & \cdots & \tilde{B}_{i+2,i+1} & \tilde{B}_{i+2,i} \\
\tilde{B}_{i+1,n} & \tilde{B}_{i+1,n-1} & \cdots & \tilde{B}_{i+1,i+1} & \tilde{B}_{i+1,i}
\end{vmatrix}.
\end{equation}
And again using the definition of $\tilde{B}_{m,k}$, we get
\begin{equation}\label{AEq4}
M_{n-i+1}^{n+1,i+1}=\begin{vmatrix}
1 & \prod_{j=1}^{2}{\frac{(j+n-2)(j+n-1)}{jn(n+1)}} & \cdots &\prod_{j=1}^{n-i}{\frac{(j+i)(j+i+1)}{jn(n+1)}} & \prod_{j=1}^{n-i+1}{\frac{(j+i-1)(j+i)}{jn(n+1)}} \\
1 & \prod_{j=1}^{1}{\frac{(j+n-2)(j+n-1)}{jn(n+1)}} & \cdots &\prod_{j=1}^{n-i-1}{\frac{(j+i)(j+i+1)}{jn(n+1)}} & \prod_{j=1}^{n-i}{\frac{(j+i-1)(j+i)}{jn(n+1)}} \\
0 & 1 & \cdots &\prod_{j=1}^{n-i-2}{\frac{(j+i)(j+i+1)}{jn(n+1)}} & \prod_{j=1}^{n-i-1}{\frac{(j+i-1)(j+i)}{jn(n+1)}} \\
\vdots & \vdots & \ddots & \vdots & \vdots \\
0 & 0 & \cdots &\prod_{j=1}^{1}{\frac{(j+i)(j+i+1)}{jn(n+1)}} & \prod_{j=1}^{2}{\frac{(j+i-1)(j+i)}{jn(n+1)}} \\
0 & 0 & \cdots &1 & \prod_{j=1}^{1}{\frac{(j+i-1)(j+i)}{jn(n+1)}}
\end{vmatrix}.
\end{equation}
Notice that $M_{n-i}^{n+1,i+1}$ and $M_{n-i+1}^{n+1,i+1}$ only differ in the first column. Thus, we deduce that the difference $M_{n-i}^{n+1,i+1}-(n-i+1)M_{n-i+1}^{n+1,i+1}$ is equal to
\begin{equation}\label{AEq5}
\begin{vmatrix}
-(n-i) & \prod_{j=1}^{2}{\frac{(j+n-2)(j+n-1)}{jn(n+1)}} & \cdots &\prod_{j=1}^{n-i}{\frac{(j+i)(j+i+1)}{jn(n+1)}} & \prod_{j=1}^{n-i+1}{\frac{(j+i-1)(j+i)}{jn(n+1)}} \\
-(n-i+1) & \prod_{j=1}^{1}{\frac{(j+n-2)(j+n-1)}{jn(n+1)}} & \cdots &\prod_{j=1}^{n-i-1}{\frac{(j+i)(j+i+1)}{jn(n+1)}} & \prod_{j=1}^{n-i}{\frac{(j+i-1)(j+i)}{jn(n+1)}} \\
0 & 1 & \cdots &\prod_{j=1}^{n-i-2}{\frac{(j+i)(j+i+1)}{jn(n+1)}} & \prod_{j=1}^{n-i-1}{\frac{(j+i-1)(j+i)}{jn(n+1)}} \\
\vdots & \vdots & \ddots & \vdots & \vdots \\
0 & 0 & \cdots &\prod_{j=1}^{1}{\frac{(j+i)(j+i+1)}{jn(n+1)}} & \prod_{j=1}^{2}{\frac{(j+i-1)(j+i)}{jn(n+1)}} \\
0 & 0 & \cdots &1 & \prod_{j=1}^{1}{\frac{(j+i-1)(j+i)}{jn(n+1)}}
\end{vmatrix}.
\end{equation}

We want to see that this determinant is zero. Let $\bfc_k$ denote the $k$-th column of the determinant, and let us show that the last column, i.e. ${\bf c}_{n-i+1}$, is a linear combination of the remaining ${\bf c}_k$. To do this we can directly compute $\lambda_1,\lambda_2,\ldots,\lambda_{n-i}$ such that ${\bf c}_{n-i+1}=\sum_{k=1}^{n-i}\lambda_k {\bf c}_k$. Indeed, imposing this we get a linear system ${\mathcal M}$ with $n-i$ unknowns (the $\lambda_k$) and $n-i+1$ equations whose coefficient matrix is the submatrix corresponding to the first $n-i$ columns of the matrix in Eq. \eqref{AEq5}, and whose equations correspond to the rows of the matrix in Eq. \eqref{AEq5}. If we pick the equations 2 to $n-i+1$ of the system ${\mathcal M}$, which correspond to the rows 2 to $n-i+1$ of the matrix in Eq. \eqref{AEq5}, we get a triangular system whose solution $\widehat{\lambda}_1,\ldots,\widehat{\lambda}_{n-i}$ can be directly computed. Then it is a lengthy, but direct exercise, to check that $\widehat{\lambda}_1,\ldots,\widehat{\lambda}_{n-i}$ also satisfy the first equation of ${\mathcal M}$, so  ${\bf c}_{n-i+1}=\sum_{k=1}^{n-i}\widehat{\lambda}_k {\bf c}_k$.     

\end{proof}

\end{document}